\documentclass[11pt]{amsart}
\usepackage{amssymb,bm,mathrsfs,xcolor}
\usepackage{mathdots}
\usepackage{tikz-cd}

\usepackage{hyperref}
\hypersetup{colorlinks=true}

\newcommand{\map}[1]{\xrightarrow{#1}}

\newcommand{\iso}{\cong}
\newcommand{\define}{\stackrel{\mathrm{def}}{=}}

\DeclareMathOperator{\Gal}{\mathrm{Gal}}
\DeclareMathOperator{\Hom}{\mathrm{Hom}}
\DeclareMathOperator{\Aut}{\mathrm{Aut}}
\DeclareMathOperator{\End}{\mathrm{End}}
\DeclareMathOperator{\Spec}{\mathrm{Spec}}
\newcommand{\Q}{\mathbb Q}
\newcommand{\Z}{\mathbb Z}

\newcommand{\C}{\mathbb C}

\newcommand{\co}{\mathcal O}

\newcommand{\action}{\bullet}

\newcommand{\LRZ}{\mathrm{atom}}

\DeclareMathOperator{\ord}{\mathrm{ord}}
\DeclareMathOperator{\Lie}{\mathrm{Lie}}

\DeclareMathOperator{\GL}{\mathrm{GL}}

\DeclareMathOperator{\Herm}{\mathrm{Herm}}

\DeclareMathOperator{\Vol}{\mathrm{Vol}}
\DeclareMathOperator{\Den}{\mathrm{Den}}

\begin{document}
\author{Benjamin Howard}
\title{Arithmetic Siegel-Weil  for the spherical Hecke algebra}
\date{}

\theoremstyle{plain}
\newtheorem{theorem}{Theorem}[subsection]
\newtheorem{proposition}[theorem]{Proposition}
\newtheorem{lemma}[theorem]{Lemma}
\newtheorem{corollary}[theorem]{Corollary}
\newtheorem{conjecture}[theorem]{Conjecture}

\newtheorem{bigtheorem}{Theorem}[section]

\theoremstyle{definition}
\newtheorem{definition}[theorem]{Definition}
\newtheorem{hypothesis}[theorem]{Hypothesis}

\theoremstyle{remark}
\newtheorem{remark}[theorem]{Remark}
\newtheorem{example}[theorem]{Example}
\newtheorem{question}[theorem]{Question}

\numberwithin{equation}{subsection}
\renewcommand{\thebigtheorem}{\Alph{bigtheorem}}

\begin{abstract}
We formulate a conjectural generalization of the arithmetic Siegel-Weil formula, relating intersection multiplicities of cycles  on unitary Rapoport-Zink spaces to central derivatives of local Whittaker functions.
The new aspect of the conjecture is that it incorporates the action of certain Hecke correspondences on the Rapoport-Zink space.
We verify the conjecture in a low-dimensional case.

A significant portion of the paper is devoted to  generalizing the classical theory of  representation densities for local Hermitian spaces, as this is needed for explicit computation of the relevant Whittaker functions. We  expect this generalization to be of independent interest.
\end{abstract}

\thanks{
This work was supported in part by NSF grant DMS-2101636.
}

\maketitle
%\setcounter{section}{-1}
%\tableofcontents

%%%%%%%%%%%%%%%%%%%%%%%%%%%

\section{Introduction}

%%%%%%%%%%%%%%%%%%%%%%%%%%%%

Inspired by the conjectural \emph{arithmetic fundamental lemma for the spherical Hecke algebra} of Li-Rapoport-Zhang \cite{LRZ},  we formulate a conjectural \emph{arithmetic Siegel-Weil formula for the spherical Hecke algebra} on unitary Rapoport-Zink formal schemes.
Our conjecture  generalizes the original arithmetic Siegel-Weil formula conjectured by Kudla-Rapoport \cite{KR1},  and proved by Li-Zhang \cite{LZ1}, by incorporating the action of certain Hecke correspondences constructed in \cite{LRZ}.

Using a suitably generalized notion of representation densities, we verify a nontrivial case of  our  conjecture by direct calculation.

%%%%%%%%%%%%%%%%%%%%%%%%%%%

\subsection{The arithmetic Siegel-Weil formula}
\label{ss:introASW}

%%%%%%%%%%%%%%%%%%%%%%%%%%%%

The (local)  \emph{arithmetic Siegel-Weil formula}, formulated by Kudla-Rapoport \cite{KR1} and proved by Li-Zhang \cite{LZ1},   is a relation between a local  arithmetic intersection number and the central derivative of a local Whittaker function.     
The precise statement is recalled below in Theorem \ref{thm:ASW}, but let us here at least sketch the idea.  

Let $F$ be a finite extension of $\Q_p$ with $p$ an odd prime, and let $F'/F$ be an unramified quadratic extension.
Fix  an integer $n\ge 2$.
Up to isometry,  there are exactly  two nondegenerate $F'$-Hermitian spaces of dimension $n$.   
We call  these $V$ and $\mathbb{V}$, and  distinguish between them using a simple criterion: $V$ admits an $\co_{F'}$-lattice that is self-dual with respect to the Hermitian form, while $\mathbb{V}$ does not.

On the arithmetic side, one has a unitary Rapoport-Zink space  $\mathcal{N}_n$  of signature $(1,n-1)$.  
It is  a formal scheme, formally smooth of relative dimension  $n-1$ over the completion of the ring of integers of the maximal unramified extension of $F'$.
Any nonzero vector $x\in \mathbb{V}$ determines a \emph{Kudla-Rapoport} divisor $\mathcal{Z}(x) \subset \mathcal{N}_n$.  More generally, for any $1 \le d \le n$ and any tuple $x=(x_1,\ldots, x_d) \in \mathbb{V}^d$ with linearly independent components, one can consider the scheme-theoretic intersection 
\[
\mathcal{Z}(x) = \mathcal{Z}(x_1) \cap \cdots \cap \mathcal{Z}(x_d) \subset \mathcal{N}_n.
\]
  This is typically not equidimensional, but one can form a suitable derived variant of the intersection to define an object, called   $ \mathcal{Z}^\mathbb{L}(x_1) \otimes^{\mathbb{L}}  \cdots \otimes^{\mathbb{L}}  \mathcal{Z}^\mathbb{L}(x_d) $,
  that effectively functions as a codimension $d$ cycle on $\mathcal{N}_n$.   For example, when $d=n$ it has  a well-defined \emph{degree} 
 \[
\chi \big( \mathcal{Z}^\mathbb{L}(x_1) \otimes^{\mathbb{L}}  \cdots \otimes^{\mathbb{L}}  \mathcal{Z}^\mathbb{L}(x_n)  \big) \in \Z .
 \]
 The precise definition involves taking the Euler-Poincar\'e characteristic of a complex of vector bundles on $\mathcal{N}_n$, which is why we use the notation $\chi$ instead of the more obvious $\deg$.

On the analytic side, to any  Schwartz function $\varphi \in S(V^n)$ and any nonsingular Hermitian matrix $T \in \Herm_n(F')$, 
one can associate a  local Whittaker function $W_T(s,\varphi)$.
This is a polynomial in $q^{-s}$, where $q$ is the residue cardinality of $F$.
One should regard these $W_T(s,\varphi)$ as  local analogues of  Fourier coefficients of Eisenstein series on a quasi-split unitary group of rank $n$.

Now  fix a self-dual lattice $\Xi \subset V$.  Fix also  an $F'$-basis $x_1,\ldots, x_n \in \mathbb{V}$, and  denote by $T  \in \Herm_n(F')$ the matrix of Hermitian pairings (the \emph{moment matrix}) of the $x_i$'s.  
For a $T$ that arises in this way,  the local Whittaker function $W_T(s,\varphi )$ vanishes at $s=0$ for \emph{every} choice of $\varphi \in S(V^n)$, and in particular vanishes for the characteristic function
\[
\mathbf{1}_\Xi \otimes \cdots  \otimes \mathbf{1}_\Xi \in S(V) \otimes \cdots \otimes S(V) \iso S(V^n) 
\]
 of $\Xi^n \subset V^n$.
The arithmetic Siegel-Weil formula of Kudla-Rapoport \cite{KR1} and Li-Zhang \cite{LZ1} is the equality
\[
\chi \big( \mathcal{Z}^\mathbb{L}(x_1) \otimes^{\mathbb{L}}  \cdots \otimes^{\mathbb{L}}  \mathcal{Z}^\mathbb{L}(x_n)  \big)  
= \frac{ W'_T(0, \mathbf{1}_\Xi \otimes \cdots  \otimes \mathbf{1}_\Xi  ) }{  c(n) \cdot   \log(q^2)    },
\]
where $q$ is the residue cardinality of $F$, and 
\[
c(n)=  \prod_{i=1}^n (   1- (-q)^{-i}  ).
\]

See the works of Li-Liu \cite{LL1, LL2} and Disegni-Liu \cite{DL} for applications of the arithmetic Siegel-Weil formula to the  Beilinson-Bloch conjecture for  unitary Shimura varieties.
Generalizations of the arithmetic Siegel-Weil formula in the presence of ramification (either in $F'/F$ or in the level structure) have been proved by He-Li-Shi-Yang \cite{HSY} and Cho-He-Zhang \cite{CHZ}.

%%%%%%%%%%%%%%%%%%%%%%%%%%%%

\subsection{Action of the spherical Hecke algebra}

%%%%%%%%%%%%%%%%%%%%%%%%%%%%

Inspired by work of Leslie \cite{Leslie}  generalizing  the Jacquet-Rallis fundamental lemma, Li-Rapoport-Zhang \cite{LRZ} formulated a conjectural extension of Zhang's arithmetic fundamental lemma \cite{Zhang1, Zhang2},
incorporating  the action of the spherical Hecke algebra of a quasi-split unitary group.  The purpose of this paper is to investigate a similar generalization of the arithmetic Siegel-Weil formula.

Keep the notation of the previous subsection, let $\mathrm{U}(V)$ be the unitary group (regarded as a locally profinite group) of the $n$-dimensional Hermitian space $V$, and denote by $K \subset \mathrm{U}(V)$ the stabilizer of the fixed self-dual lattice $\Xi \subset V$.  Let $\mathcal{H}(V)$ be the spherical Hecke algebra of $\Q$-valued $K$-bi-invariant compactly supported functions on $\mathrm{U}(V)$.    
There  is a natural action of $\mathcal{H}(V)$ on the space $S(V)^K$ of $K$-fixed Schwartz functions on $V$, and so it makes sense to form
\[
( f_1\mathbf{1}_\Xi) \otimes \cdots  \otimes ( f_n\mathbf{1}_\Xi  )    \in S(V) \otimes \cdots \otimes S(V) \iso S(V^n) 
\]
for any $f_1,\ldots, f_n \in \mathcal{H}(V)$.

The spherical Hecke algebra is isomorphic to a polynomial ring in $\lfloor n/2 \rfloor$ variables.  In fact, Li-Rapoport-Zhang \cite{LRZ} construct a preferred set of polynomial generators $f_t^\LRZ \in \mathcal{H}(V)$,  indexed by even integers $2 \le t \le n$.  
For each of these \emph{atomic Hecke functions}, Li-Rapoport-Zhang construct a geometric analogue: an \emph{atomic Hecke correspondence} $\mathbb{T}^\LRZ_t$ that (morally, but not quite literally) acts as an endomorphism of the group of cycles on $\mathcal{N}_n$.  
They conjecture that these atomic Hecke correspondences pairwise commute, so that the Hecke algebra $\mathcal{H}(V)$ acts on cycles via the map $f_t^\LRZ \mapsto \mathbb{T}^\LRZ_t$.
 This allows one to form, at least assuming the commutativity conjecture, a  codimension $n$ cycle 
\[
( f_1\mathcal{Z}^\mathbb{L}(x_1)  ) \otimes^{\mathbb{L}}  \cdots \otimes^{\mathbb{L}}  (  f_n \mathcal{Z}^\mathbb{L}(x_n) )
\]
on $\mathcal{N}_n$ for any basis $x_1,\ldots, x_n \in \mathbb{V}$ and any $f_1,\ldots, f_n \in \mathcal{H}(V)$.

We conjecture  that the arithmetic Siegel-Weil formula recalled  above is compatible with these Hecke actions, in the sense that 
\begin{equation}\label{intro conjecture}
\chi \big( f_1 \mathcal{Z}^\mathbb{L}(x_1) \otimes^{\mathbb{L}}  \cdots \otimes^{\mathbb{L}}  f_n\mathcal{Z}^\mathbb{L}(x_n)  \big)  
\stackrel{?}{=}  \frac{ W'_T(0, f_1\mathbf{1}_\Xi \otimes \cdots  \otimes f_n\mathbf{1}_\Xi  ) }{  c(n) \cdot   \log(q^2)    },
\end{equation}
where the $T \in \Herm_n(F')$ on the right is the matrix of Hermitian inner products of the $x_i$'s.
A more general conjecture, formulated without assuming the commutativity of the atomic Hecke correspondences, is stated in the text as Conjecture \ref{conj:sphericalASW}.

As evidence toward the conjectural equality \eqref{intro conjecture}, we prove the following special case.  
Note that when $n=2$ there is only one atomic Hecke correspondence $\mathbb{T}^\LRZ_2$, so the commutativity conjecture of Li-Rapoport-Zhang is vacuously true.

\begin{bigtheorem}\label{Thm A}
Suppose $n=2$.  Let  $x_1,x_2\in \mathbb{V}$ be any $F'$-basis for which the Hermitian norms of $x_1$ and $x_2$ lie in $\co_F^\times$. The equality
\[
\chi \big( f_1 \mathcal{Z}^\mathbb{L}(x_1) \otimes^{\mathbb{L}}   f_2\mathcal{Z}^\mathbb{L}(x_2)  \big)  
=  \frac{ W'_T(0, f_1\mathbf{1}_\Xi \otimes  f_2\mathbf{1}_\Xi  ) }{  c(2) \cdot   \log(q^2)    }
\]
holds for all $f_1,f_2 \in \mathcal{H}(V)$,  where  $T \in \Herm_2(F')$  denotes the matrix of Hermitian pairings of $x_1$ and $x_2$.
\end{bigtheorem}

The proof of Theorem \ref{Thm A} is by explicit calculation of both sides. 
The calculation of the intersection multiplicity on the left-hand side reduces  quickly to known formulas of  Kudla-Rapoport \cite{KR1} and Li-Rapoport-Zhang \cite{LRZ}.  
 Most of this paper is devoted to the  calculation of the right-hand side, which turns out to be rather involved.

If $f_1=f_2=1$,  then we are in the $n=2$ case of the original arithmetic Siegel-Weil formula.  
This case was already proved by Kudla-Rapoport \cite{KR1}, who reduced   the calculation of  $W_T( s ,    \mathbf{1}_\Xi \otimes  \mathbf{1}_\Xi )$ to known formulas for representation densities of $T$ by Hermitian lattices.

To go beyond the case $f_1=f_2=1$,  one needs a more general  theory of representation densities.   
 Although there are  hints in the literature pointing in the right direction, for example    \cite{CY, Cho},  the necessary theory does not seem to have been explored in a systematic way.
 In the next subsection, we explain what sort of generalized  representation densities we study.

%%%%%%%%%%%%%%%%%%%%%%%%%%%%

\subsection{Representation densities}
\label{ss:intro density}

%%%%%%%%%%%%%%%%%%%%%%%%%%%%

Now let $F'/F$ be any  separable quadratic extension of local fields (possibly ramified, and $\mathrm{char}(F)=2$ is allowed).
As above, let $(V,h)$ be a nondegenerate Hermitian space over $F'$ of dimension $n$.  
For a tuple $x=(x_1,\ldots, x_d) \in V^d$, we denote by 
\[
\tau(x) \in \Herm_d(F')
\]
 the  moment matrix of $x$, whose $(i,j)$-entry is the Hermitian pairing $h(x_i,x_j)$.

For a nonsingular Hermitian matrix $T\in \Herm_d(\co_{F'})$ and an $\co_{F'}$-lattice $M \subset V$ on which the Hermitian form is $\co_{F'}$-valued, the classical representation density of $T$ by $M$ is meant to measure  the density of the subset 
\[
\{ x   \in M^d : \tau(x) = T \} \subset M^d.
\]
This density is $0$ if interpreted literally, so the actual definition is more involved.
Fix a uniformizer $\pi \in \co_F$, and for any $k >0$ define 
\[
A_{\pi^k}  (  M  , T )
= \left\{ x\in ( M/\pi^kM)^d  : \tau( \widetilde{x} ) \equiv  T \mod \pi^k \Herm_d(\co_{F'})^\vee  \right\}.
\]
Here $\widetilde{x} \in M^d$ is any lift of $x$, and 
\[
 \Herm_d(\co_{F'})^\vee
 = \left\{ t \in \Herm_d(F') : \mathrm{Tr}(ts) \in \co_F , \forall \, s \in \Herm_d(\co_{F'}) \right\}.
\]
The classical representation density of  $T$ by $M$ is defined, as in   \cite[(3.1)]{Hir} or  \cite[(10.1)]{KR2}, as 
\[
\alpha(M,T)  = 
 \frac{ \# A_{\pi^k}  ( M , T  ) }{  q^{  k d (2  n - d ) }  } 
\]
for any $k \gg 0$.     Here $q$ is the residue cardinality of $F$.

The motivation for this  definition of $\alpha(M,T)$  is a practical one: it is essentially a special value of a Whittaker function $W_T(s, \mathbf{1}_{M^d})$ associated to the Schwartz function  $\mathbf{1}_{M^d} \subset S(V^d)$, and one can use an interpolation trick to express other values of this function also as classical representation densities. 
Thus the calculation of $W_T(s, \mathbf{1}_{M^d})$ is reduced to the problem of computing classical representation densities.

Whittaker functions of the type  appearing  in the conjectural equality  \eqref{intro conjecture} are more complicated, being  linear combinations of Whittaker functions of the form  $W_T(s, \mathbf{1}_{M_1} \otimes \cdots \otimes \mathbf{1}_{M_d})$ for  $\co_{F'}$-lattices $M_1,\ldots, M_d \subset V$.
Thus one is naturally led to study the density of subsets of the form
  \[
\{ x   \in M_1 \times \cdots \times M_d : \tau(x) = T \} \subset  M_1 \times \cdots \times M_d ,
\]
and more generally of the form 
\[
\{ x \in C : \tau(x) = T \} \subset C
\]
for \emph{any} compact open subset $C \subset V^d$.  
 As above, these densities are $0$ if interpreted literally.
A substantial part of this paper is dedicated to generalizing the definition of representation densities to cover this sort of situation, and to developing enough computational techniques to find an explicit formula for the particular representation densities relevant to the Whittaker function appearing in Theorem \ref{Thm A}.
 
 In \S \ref{ss:representation densities}, we attach to a nonsingular $T\in \Herm_d(F')$ and a nonempty compact open subset $C \subset V^d$,  a representation density $ \Den_T(C) \in \Q$.  The simplest definition to state is
\[
\Den_T(C) =
  \frac{   \Vol ( \{ x\in  C  : \tau(x)   \in U  \}  ) }{   \Vol( C )   }  \cdot \frac{ \Vol( \Herm_d(\co_{F'}) ^\vee )  }{ \Vol(U) } ,
\]
where $U \subset \Herm_d(F')$ is any sufficiently small compact open neighborhood of $T$.
The first factor on the right-hand side is a quotient of volumes in  $V^d$,  the second  is a quotient of volumes in  $\Herm_d(F')$, and all volumes are computed with respect to Haar measures on these $F$-vector spaces.
One must prove that this quantity is independent of the sufficiently small $U$, and for this it turns out to be easiest to  first formulate a different definition (Definition \ref{def:density}) of $\Den_T(C)$, and then prove  that it agrees with the right-hand side above for any sufficiently small $U$.

 In the special case where $C=M^d$ for an $\co_{F'}$-lattice, our definition  recovers the classical representation density
 \[
 \Den_T(M^d) = \alpha(M,T).
 \] 
 
A warning: if  the nonsingular Hermitian matrix $T$ has entries in $\co_{F'}$, then one can use it to endow $(\co_{F'})^d$ with the structure of a  Hermitian lattice.  
The classical representation density $\alpha(M,T)$ only depends on the isometry class of this Hermitian lattice, and this allows one to reduce the calculation of $\alpha(M,T)$ to particularly simple choices of $T$.  For example, if one excludes  the case in which  $F'/F$ is ramified with residue characteristic $2$, then one may assume that $T$ is diagonal.
This does not work for our more general representation densities; for the choices of $C \subset V^d$ of interest to us, $\Den_T(C)$ genuinely depends on the matrix $T$, not just the isometry class of the associated Hermitian lattice.

 In \S \ref{ss:density polynomial} we construct a \emph{density polynomial}  
 \[
 \Den_T(s,C) \in \Q[q^{-s}],
 \]
  whose value at $s=0$ is  $\Den_T(C)$.
 For any $\varphi \in S(V^d)$, we show that the associated Whittaker function $W_T(s, \varphi)$ can be expressed as a linear combination of these density polynomials, by proving  in \S \ref{ss:whittaker}  that $\Den_T(s,C)$ is essentially  $W_T(s, \mathbf{1}_C)$.

 In the more technical  \S \ref{s:tools} and \S \ref{s:density formulas} we  develop some general tools for computing $\Den_T(C)$, and then apply these to compute the particular density polynomials that are relevant to the calculation of the right-hand side of the equality of Theorem \ref{Thm A}.  
 
 With these calculations in hand, in \S \ref{s:ASW}  we finally turn to the task outlined in the previous subsection:  formulating a conjectural version of the arithmetic Siegel-Weil formula that incorporates the action of the spherical Hecke algebra, and proving the special case stated as Theorem \ref{Thm A}.
 
%%%%%%%%%%%%%%%%%%%%%%%%%%%

\subsection{Notation}

%%%%%%%%%%%%%%%%%%%%%%%%%%%%

Throughout the paper,  $F'/F$ is a separable quadratic extension of local fields, and $\sigma \in \Gal(F'/F)$ is the nontrivial automorphism.  
We denote by  $\pi  \in \co_F$  a uniformizing parameter, and by $q$   the cardinality of the residue field $\co_F/\pi\co_F$ .   In  \S \ref{s:density formulas} and \S \ref{s:ASW}  we specialize to the case of an unramified quadratic  extension of finite extensions of  $\Q_p$ with $p$ odd.

If $R$ is an $\co_F$-algebra, we often abbreviate
\[
R' = R \otimes_{\co_F} \co_{F'} ,
\]
and by mild abuse of notation write $\sigma = \mathrm{id} \otimes \sigma \in \Aut(R'/R)$.
Let  $\Herm_d(R')$ be  the $R$-module of $d\times d$ Hermitian matrices valued in $R'$.

If $(V,h)$ is a finite dimensional Hermitian space over $F'$, and  $R$ is  an $F$-algebra, then
\[
V_R= V\otimes_F R
\]
is naturally a Hermitian module  over  the ring $R'$.
Hermitian pairings  are always linear in the first variable, and $\sigma$-linear in the second.

%
%If $W$ is a Hermitian space for $F'/F$, then 
%\[
%W_R = W \otimes_F R
%\]
%is naturally a Hermitian $R'$-module, and we define the  the \emph{moment matrix} of a $d$-tuple $x=(x_1,\ldots, x_d) \in W_R^d$ as the matrix of Hermitian inner products
%\[
%\tau(x)  \define  (  h(x_i,x_j) )_{1 \le i,j\le d}  \in  \Herm_d(R') .
%\]
%
%

%Up to scaling, any finite dimensional $F$-vector space $W$ has  a unique translation-invariant measure.  
%For nonempty compact open subsets  $C_1,C_2 \subset W$, the quotient of their volumes
%$ \Vol(C_1)  /   \Vol(C_2) $ is  therefore a  well-defined rational number, independent of any choices.  
%For example if $C_2 \subset C_1$ are $\co_F$-lattices in $W$, then
%$
% \Vol(C_1) /  \Vol(C_2) = \#(C_1/C_2).
%$

%%%%%%%%%%%%%%%%%%%%%%%%%%%

\section{Generalized representation densities}
%\label{s:general density}

%%%%%%%%%%%%%%%%%%%%%%%%%%%%

Let $(V,h)$ be a finite dimensional Hermitian space over  $F'$, and  fix an integer $d\ge 1$.

Our goal in this section is to define, for  any nonsingular Hermitian matrix $T \in \Herm_d(F')$,  and any compact open subset $C \subset V^d$,   a generalized representation density $\Den_T(C) \in \Q$.   
We then extend the definition by constructing  a density polynomial 
$
\Den_T(s, C ) \in \Q [ q^{-s}],
$
and explain its connection   to   Whittaker functions for principal series representations of the  quasi-split unitary group $\mathrm{U}(d,d)$ over $F$.

%%%%%%%%%%%%%%%%%%%%%%%%%%%

\subsection{The moment morphism}

%%%%%%%%%%%%%%%%%%%%%%%%%%%%

Our initial approach to representation densities will be algebro-geometric in nature.
Let $X=V^d$ and $Y=\Herm_d(F')$, but regarded as $F$-schemes.  
In other words, for any $F$-algebra $R$, the functors of points are 
\begin{equation}\label{XandY}
X(R) = V_R^d \qquad \mbox{and} \qquad Y(R)=\Herm_d( R' ).
\end{equation}

  Any  $d$-tuple $x=(x_1,\ldots, x_d) \in V_R^d$  has a 
\emph{moment matrix}
 \begin{equation}\label{moment matrix}
\tau(x)  =  (  h(x_i,x_j) )_{1 \le i,j\le d}  \in  \Herm_d(R') ,
\end{equation}
whose entries are just the Hermitian pairings of the components of $x$.
    The \emph{moment morphism} 
\begin{equation}\label{moment morphism}
\tau : X\to Y
\end{equation}
is the  morphism of $F$-schemes whose induced map on $R$-points
 sends  $x\in X(R)$ to its moment matrix \eqref{moment matrix}.

The proof of the following lemma  is essentially the same as that of \cite[Lemma 5.5.2]{GY} and  \cite[Lemma 3.10]{CY}.
We repeat it here, as variants of the argument will be needed later.

\begin{lemma}\label{lem:tangent moment}
Suppose $\kappa$ is an algebraically closed  field containing  $F$. 
The moment morphism $ \tau : X_{\kappa} \to Y_{\kappa}$ is smooth at every point 
\[
x=(x_1,\ldots, x_d) \in  V_\kappa^d = X(\kappa) 
\]
 for which the $ \kappa'$ submodule generated by $x_1,\ldots, x_d \in V_\kappa$ is free of rank $d$.
\end{lemma}

\begin{proof}
The  map  on tangent spaces
\begin{equation}\label{tangent map}
V_\kappa^d   \iso  T_x X_\kappa     \to    T_{\tau(x)} Y_\kappa \iso    \Herm_d(\kappa')
\end{equation}
 induced by the moment morphism is 
\[
 (v_1,\ldots, v_d)   \mapsto  \big( h(x_i,v_j)  + h(v_i,x_j)   \big)_{ 1\le i , j \le d }  .
\]
In order to show  that  \eqref{tangent map}  is surjective, we factor it as the composition
\begin{equation}\label{tangent factor}
V_\kappa^d \to \mathrm{Mat}_{d \times d} (\kappa')  \to \Herm_d(\kappa')
\end{equation}
in which the first arrow is 
\[
(v_1, \ldots, v_d) \mapsto  \big( h(x_i,v_j)  \big)_{ 1\le i , j \le d },
\] 
and the second is $A\mapsto A+  {}^tA^\sigma$.  

If  $W \subset V_\kappa$ denotes the $ \kappa'$-submodule of elements perpendicular to all  $x_1,\ldots, x_d$, 
our assumption on $x$ implies that $W$ is free over $\kappa'$, and satisfies 
\[
\mathrm{rank}_{\kappa'}(W)  =  \mathrm{rank}_{\kappa'}(V_\kappa)   -d  .
\]
The kernel of the first arrow in \eqref{tangent factor} is $W^d \subset V_\kappa^d$, and one checks that this arrow is surjective by comparing $\kappa$-dimensions of the kernel, source, and target.
The second arrow in \eqref{tangent factor} is easily seen to be surjective, using  surjectivity of the trace map $\kappa'\to \kappa$.

To complete the proof, note that the condition imposed on $x\in X(\kappa)$ in the statement of the lemma is a Zariski open condition.
If $U \subset X_\kappa$ is the largest open subset on which this condition holds, then $\tau : U \to Y_\kappa$ is a morphism between smooth $\kappa$-varieties, inducing  surjections on tangent spaces.  It is therefore smooth. 
\end{proof}

Denote by $Y^\mathrm{ns} \subset Y$ the open subscheme of Hermitian matrices with invertible determinant, and by $X^\mathrm{ns} =\tau^{-1}(Y^\mathrm{ns}) \subset X$ its preimage under the moment morphism.  
It follows from Lemma \ref{lem:tangent moment} that the moment morphism restricts to a smooth map
\begin{equation}\label{smooth moment}
\tau : X^\mathrm{ns}  \to Y^\mathrm{ns}.
\end{equation}

Being schemes of $F$-vector spaces,  $X$ and $Y$ admit nowhere vanishing translation-invariant  top-degree differential forms 
\begin{equation}\label{mewtwo}
\mu_X \in H^0( X , \Omega_X^{\dim(X)} ) 
\qquad\mbox{and}\qquad
\mu_Y \in H^0( Y , \Omega_Y^{\dim(Y)} ) .
\end{equation}
These are unique up to scaling by $F^\times$.  
As in \cite[\S 3.2]{GY} and \cite[\S 3.1]{CY}, the restrictions 
$
\mu_{X^\mathrm{ns}} = \mu_X|_{X^\mathrm{ns}}
$
and
$
\mu_{Y^\mathrm{ns}} = \mu_Y|_{Y^\mathrm{ns}}
$
  determine a nowhere vanishing section
\begin{equation}\label{quotient form}
 \frac{ \mu_{X^\mathrm{ns} } } {  \tau^*\mu_{Y^\mathrm{ns} }}
 \in H^0( X^\mathrm{ns} , \Omega_{X^\mathrm{ns} / Y^\mathrm{ns} }^{\dim(X)-\dim(Y)} ) .
\end{equation}
Indeed,  smoothness of \eqref{smooth moment} implies the exactness of the sequence
\[
0 \to \tau^* \Omega^1_{ Y^\mathrm{ns} } \to \Omega^1_{ X^\mathrm{ns} } \to \Omega^1_{ X^\mathrm{ns} / Y^\mathrm{ns} } \to 0 
\]
of vector bundles on $X^\mathrm{ns}$, 
and taking top exterior powers yields a canonical isomorphism of line bundles
\[
\Omega^{\dim(X)}_{ X^\mathrm{ns} }  \iso   \Omega^{ \dim(X) - \dim(Y) } _{ X^\mathrm{ns} / Y^\mathrm{ns} }  \otimes  \tau^* \Omega^{\dim(Y)}_{ Y^\mathrm{ns} }  .
\]
The form \eqref{quotient form} is defined by the relation
\[
\mu_{X^\mathrm{ns}}   =  \frac{ \mu_{X^\mathrm{ns} } } {  \tau^*\mu_{Y^\mathrm{ns} }} \otimes \tau^*\mu_{Y^\mathrm{ns} }.
\]

A nonsingular Hermitian matrix 
$
T \in  Y^\mathrm{ns}(F) 
$
determines a smooth $F$-scheme $X_T$, defined by the cartesian diagram 
\[
\begin{tikzcd}
 { X_T }\ar[r] \ar[d]  &  { X^\mathrm{ns} } \ar[d , "\tau" ]   \\
{ \Spec(F) } \ar[r, "T" ' ]  & { Y^\mathrm{ns} . }
\end{tikzcd}
\]
In other words, $X_T$ is the smooth variety with $F$-points
\begin{equation}\label{X_T points}
X_T(F) = \{ x\in V^d : \tau (x)  = T\}.
\end{equation}
The form  \eqref{quotient form}  pulls back   to a  nowhere vanishing top-degree form 
\begin{equation}\label{mewT}
\mu_{X_T  }   \define    \frac{ \mu_{X^\mathrm{ns} } } { \tau^* \mu_{Y^\mathrm{ns} }}\Big|_{ X_T }   \in H^0(X_T , \Omega_{ X_T  }^{\dim(X_T) } ) .
\end{equation}
This differential form will be essential to our definition of generalized representation densities in the next subsection.

%%%%%%%%%%%%%%%%%%%%%%%%%%%%%%%%%%%%%

\subsection{Representation densities}
\label{ss:representation densities} 

%%%%%%%%%%%%%%%%%%%%%%%%%%%%%%%%%%%%%

For any smooth $F$-variety $U$ we may regard $U(F)$ as an \emph{$F$-analytic manifold} in the sense of \cite[\S 2.4]{Igusa}.
In particular, $U(F)$ is a totally disconnected, locally compact, Hausdorff topological space.  

As explained in \cite[Chapter 7.4]{Igusa}, the top-degree forms  $\mu_X$ and $\mu_Y$  fixed in \eqref{mewtwo} induce measures, denoted the same way,  on the $F$-analytic manifolds 
\[
X( F ) = V^d  \qquad\mbox{and} \qquad   Y( F ) = \Herm_d(F' ).
\]
These are just the usual translation-invariant Haar measures on these topological vector spaces.  
Less trivially, for any nonsingular $T \in \Herm_d(F')$ the top-degree form
 $\mu_{X_T}$ from \eqref{mewT} induces  a measure on the $F$-analytic manifold 
$X_T(F)$ from \eqref{X_T points}.
These three measures are related by the following Fubini-style formula.

\begin{proposition}\label{prop:fubini}
If   $f : X^\mathrm{ns}(F) \to \C$ is a locally constant compactly supported function, then
\[
\Phi_f(T) \define  \int_{X_T(F)}  f(x) \, d\mu_{X_T}(x)
\]
is  a locally constant compactly supported function  of   $T \in Y^\mathrm{ns}(F)$, and
\begin{equation*}
\int_{ X (F) } f(x) \, d\mu_X (x) =
\int_{ Y(F) } \Phi_f( T ) \, d\mu_Y  ( T ).
\end{equation*}
\end{proposition}

\begin{proof}
This is a consequence of \cite[Theorem 7.6.1]{Igusa}.
\end{proof}

\begin{definition}\label{def:density}
The \emph{representation density} of a nonsingular Hermitian matrix $T\in \Herm_d(F')$ with respect to a nonempty compact open subset $C \subset X(F)$ is the rational number
\[
\Den_T(C) =
 \mu_{ X_T  }(    C\cap X_T( F)  ) \cdot \frac{ \mu_Y(  \Herm_d(\co_{F'})^\vee  ) }{ \mu_X(C) } ,
\]
where $\Herm_d(\co_{F'})^\vee$ is as  in \S \ref{ss:intro density}.
For convenience, we set $\Den_T(\emptyset) =0$.
\end{definition}

\begin{remark}
Because $\mu_{X_T}$ was defined as the quotient \eqref{mewT},
the representation density $\Den_T(C)$ is independent of the choice of translation-invariant forms \eqref{mewtwo}.  
\end{remark}

\begin{remark}\label{rem:no represent}
If there is no $x \in C$ with $\tau(x)=T$, then  $\Den_T(C) =0$.
\end{remark}

The following proposition serves as a useful alternative definition  of representation densities.

\begin{proposition}\label{prop:natural density}
The representation density $\Den_T(C)$ is  locally constant  as a function of $T \in  Y^\mathrm{ns}(F)$. 
For all sufficiently small compact open neighborhoods $U \subset Y^\mathrm{ns}(F)$ of $T$ we have
\begin{equation}\label{general density}
\Den_T(C) =
  \frac{   \Vol ( \{ x\in  C  : \tau(x)   \in U  \}  ) }{   \Vol( C )   }  \cdot \frac{ \Vol( \Herm_d(\co_{F'}) ^\vee )  }{ \Vol(U) } .
\end{equation}
The first  volume quotient is taken in  $X(F)=V^d$,  the second is taken  in  $Y(F)=\Herm_d(F')$, and both are with respect to the translation-invariant Haar measures on these $F$-vector spaces.
\end{proposition}

\begin{proof}
Let $\bm{1}_U$ be the characteristic function of a   compact open subset $U \subset Y^\mathrm{ns}(F)$.
If we apply Proposition \ref{prop:fubini}  with $f$ equal to the characteristic function of 
$
C \cap  \tau^{-1}(U) \subset X^\mathrm{ns}(F), 
$
we find that 
\begin{equation}\label{good den 1}
\Phi_f(T)   =  \mu_{X_T} (   C \cap X_T(F)  )  \cdot \bm{1}_U(T)   
\end{equation}
is a locally constant function of $T \in Y^\mathrm{ns}(F)$, and that 
\begin{equation}\label{good den 2}
\mu_X  ( C \cap   \tau^{-1}(U)  )     =   \int_{ U  }  \mu_{X_T} (   C \cap X_T(F)  ) \, d\mu_Y  ( T ) .
\end{equation}

Because the  right-hand side of \eqref{good den 1} is a locally constant function of $T \in  Y^\mathrm{ns}(F)$ for \emph{every} choice of   $U$, it follows that 
 $\mu_{X_T} (   C \cap X_T(F)  )$ is also a locally constant  function of $T \in  Y^\mathrm{ns}(F)$.
 
If  $U$ is a compact open subset small enough that  $\mu_{X_T} (   C \cap X_T(F)  )$  is independent of $T\in U$,   then  for any such $T$ the equality \eqref{good den 2} simplifies to 
\[
 \mu_X  ( C \cap   \tau^{-1}(U)  )      = 
 \mu_{X_T} (   C \cap X_T(F)  )  \cdot   \mu_Y(U)  .
\]
This allows us to rewrite Definition \ref{def:density}  as
\[
\Den_T(C) =
\frac{ \mu_X  ( C \cap   \tau^{-1}(U)  )    } { \mu_X(C)   }   \cdot \frac{ \mu_Y(  \Herm_d(\co_{F'})^\vee  ) }{  \mu_Y(U)  } ,
\]
which  is equivalent to \eqref{general density}.
\end{proof}

%\begin{proposition}
%Suppose the compact open subset $C \subset V^d$ has the property that the moment matrix of every $x \in C$ is nonsingular.  The function $T \mapsto \Den_T(C)$, initially defined only for nonsingular $T \in \Herm_d(F')$, extends uniquely to a locally constant compactly supported function on all of $\Herm_d(F')$.
%This extension satisfies $\Den_T(C) =0$ whenever $\det(T)=0$.
%\end{proposition}
%
%\begin{proof}
%
%\end{proof}

%\begin{remark}
%The volume forms $\mu_X$ and $\mu_Y$ in \eqref{mewtwo} are only defined up to scaling by $F^\times$.
%But because \eqref{mewT} is defined as their  ratio,  the representation density $\Den_T(C)$ is independent of these choices.
%Of course this is also obvious  from \eqref{general density}.
%\end{remark}

\begin{remark}\label{rem:general disjoint}
If $C=C_1 \sqcup \cdots \sqcup C_k\subset V^d$ is a disjoint union  of nonempty compact open subsets then
\[
\Den_T(C) = \Den_T(C_1) \cdot \frac{\Vol(C_1)}{\Vol(C)}  + \cdots + \Den_T(C_k) \cdot \frac{\Vol(C_k)}{\Vol(C)} .
\]
%The volume quotients here are understood as in \eqref{volume quotient}.
\end{remark}

\begin{remark}\label{rem:classical density}
Fix a nonsingular Hermitian matrix $T\in \Herm_d(\co_{F'})$, and   an $\co_{F'}$-lattice  $M\subset V$
 on which the Hermitian form is $\co_{F'}$-valued.
Recall that  $\pi \in \co_F$ is  a uniformizer.   Taking  $C=M^d \subset V^d $ and 
\[
U= T+\pi^k \Herm_d(\co_{F'})^\vee
\] 
in Proposition \ref{prop:natural density} shows that 
$
\Den_T(M^d)  = \alpha(M,T) ,
$
where the right-hand side is the   classical representation density   defined in \S \ref{ss:intro density}.
\end{remark}

%%%%%%%%%%%%%%%%%%%%

\subsection{Further properties of the density}

%%%%%%%%%%%%%%%%%%%%%%

 Given $g \in \GL_d(F')$ and  $x \in V^d$, we define $x g \in V^d$ using the usual rule for multiplying a row vector by a matrix.  
Define a right action of  $\GL_d(F')$  on  $\Herm_d(F')$ by 
\begin{equation}\label{change of basis}
T \action g = {}^t g\cdot T \cdot \sigma(g),
\end{equation}
so that the moment morphism $\tau : V^d \to \Herm_d(F')$ satisfies 
\begin{equation}\label{moment shift}
\tau( x g) =  \tau(x) \action g.
\end{equation}

The representation density of Definition \ref{def:density} satisfies an invariance property under these $\GL_d(F')$-actions.

\begin{proposition}\label{prop:density invariance}
Let $C \subset V^d$ be a compact open subset.
For any $g\in \GL_d( F' )$ and  nonsingular $T \in \Herm_d(F')$, we have
\[
\Den_{T \action g} ( C g  )  = \Den_T(C )   \cdot q^{d \ord_F(\Delta)},
\]
where $q$ is the residue cardinality of $F$, and $\Delta=\mathrm{Nm}_{F'/F}( \det(g))$.
\end{proposition}

\begin{proof}
If  $U \subset  \Herm_d(F')$ is   any compact open neighborhood of $T$ then 
\begin{align*}
  \{ x\in  C  g  : \tau(x)   \in U \action g \}   
  =
   \{ x\in  C   : \tau(x  )   \in U  \}  \cdot g 
 \end{align*}
 by \eqref{moment shift}.  Taking $U$ to be a sufficiently small neighborhood of $T$,   Proposition \ref{prop:natural density} implies the equality
 \[
 \Den_{T\action g} (C g ) = \Den_T( C ) \cdot \frac{ \Vol(U) }{ \Vol(U \action g) },
 \]
 and the claim follows easily.
\end{proof}

We will use the invariance property of Proposition \ref{prop:density invariance} to give an estimate for how small the compact open subset $U \subset Y^\mathrm{ns}(F)$ in Proposition \ref{prop:natural density} needs to be for \eqref{general density} to hold.   First, a lemma.

\begin{lemma}\label{lem:open orbit lemma}
Abbreviate $K_0 = \GL_d(\co_{F'})$.
\begin{enumerate}
\item
For any compact open subset $C \subset V^d$  there exists a compact open subgroup $K \subset K_0$ such that $C g=  C$ for all $g\in K$.
\item
Every  compact open subgroup  $K \subset K_0$ acts with  open orbits on the set of nonsingular matrices in $\Herm_d(F')$.   
\end{enumerate}
\end{lemma}

\begin{proof}
For the first claim,  express $C$ as a finite union of  compact open subsets of the form $C_i=v_i+ \Lambda_i^d$,  where $v_i\in V^d$ and $\Lambda_i \subset V$ is an $\co_{F'}$-lattice.  In particular  $\Lambda_i^d  g = \Lambda_i^d$ for every $g\in K_0$.
Now fix any nonzero $m \in \co_{F'}$ such that $m v_i \in \Lambda_i^d$ for all indices $i$,  and take $K \subset K_0$ to be the principal congruence subgroup of matrices reducing to the identity modulo $m \co_{F'}$.  

For the second claim, any nonsingular $T \in \Herm_d(\co_{F'})$ endows $L_T=\co_{F'}^d$ with the structure of a Hermitian lattice.  
For all $T'$ in some compact open neighborhood of $T$, one can find an isometry $L_T \iso L_{T'} $  (for quadratic forms this is a special case  of \cite[Corollary 5.4.2]{Kitaoka}, and the Hermitian case is entirely similar), which implies that $T' \in  T \action K_0$.  
This proves that  every $K_0$-orbit of nonsingular matrices in $\Herm_d(\co_{F'})$ is open,  and the same is true for nonsingular orbits in $\Herm_d( F' )$  by scaling.  

Now let  $K \subset K_0$ be any compact open subgroup.  Given a nonsingular $T \in \Herm_d(F')$, we can write
\[
T \action K_0 = (T  \action K) \sqcup  (T_1 \action K)  \sqcup \cdots (T_\ell \action K)
\]
as a finite disjoint union of $K$-orbits.  Each of these $K$-orbits is compact, hence each is closed in $T \action K_0$, hence each is open in $T \action K_0$.  Having already proved that  $T \action K_0$ is open in $\Herm_d(F')$, the same is true of $T\action K$.
\end{proof}

\begin{proposition}\label{prop:sufficiently small}
If   $K \subset \GL_d(\co_{F'})$ is any compact subgroup  as in the first claim  of Lemma \ref{lem:open orbit lemma}, then  \eqref{general density} holds for any compact open neighborhood  $U \subset Y^\mathrm{ns}(F)$ of $T$ contained in the  orbit 
$
T\action K \subset Y^\mathrm{ns}(F).
$
(Such a $U$ exists by the second claim of   Lemma \ref{lem:open orbit lemma}.)
\end{proposition}

\begin{proof}
For any $T' \in U$ we can write $T' = T \action g$ for some $g\in K$.
Using Proposition \ref{prop:density invariance},  we see   that 
\[
\Den_{T'}( C ) = \Den_{ T\action g} (Cg ) =  \Den_T(C)
\]
is independent of $T' \in U$.  
By examining the proof of Proposition \ref{prop:natural density}, we see that this is enough to guarantee the equality \eqref{general density}.
\end{proof}

%%%%%%%%%%%%%%%%%%%%%%%%%%%%%%%

\subsection{The density polynomial}
\label{ss:density polynomial}

%%%%%%%%%%%%%%%%%%%%%%%%%%%%%%%

For any integer $r\ge 0$, endow 
\[
W_{r,r} \define (F')^{2r}
\]
  with the Hermitian form determined by the antidiagonal matrix
\[
\begin{pmatrix}  & & 1  \\ & \iddots &  \\ 1 & &   \end{pmatrix}  \in \Herm_{2r}(F').
\]
In other words, $W_{r,r}$ is the orthogonal direct sum of $r$  hyperbolic planes.
Define a self-dual  $\co_{F'}$-lattice  
$
\Lambda_{r,r} = (\co_{F'})^{2r} \subset W_{r,r}.
$

Fix compact open subset $C \subset V^d$.  For any integer  $r \ge 0$, we form the orthogonal direct sum 
$
V^{ [r] } = V \oplus W_{r,r}, 
$
and define 
\[
C^{ [r] } = C \times \Lambda_{r,r}^d \subset V^d \times W_{r,r}^d= (V^{[r]})^d.
\]

Recall that $q$ is the residue cardinality of $F$.

\begin{proposition}\label{prop:density polynomial}
 For any nonsingular $T \in \Herm_d(F')$,  there is a  polynomial $\Den_T(s,C) \in \C[q^{-s}]$  satisfying  
\[
\Den_T( r ,C) = \Den_T( C^{[r]} )
\]
 for all $r\in \Z_{\ge 0}$.   If $F'/F$ is unramified, then in fact  $\Den_T(s,C) \in \C[ q^{-2s}]$.
\end{proposition}

\begin{proof}
%Let $\pi$ be a uniformizer of $F$.
Fix  an unramified  additive character $\psi : F \to \C^\times$,  define a pairing
\[
\langle - , - \rangle_\psi :  \Herm_d(F') \times  \Herm_d(F') \to \C^\times
\]
by $\langle S,T\rangle_\psi  = \psi(  \mathrm{Tr}(ST))$, and  for  $S \in \Herm_d(F')$ set
\begin{equation}\label{f_S}
f _S(C)  = \frac{1}{\Vol(C)}   \int_C   \langle S ,    \tau(x)    \rangle_\psi   \, dx.
\end{equation}
It is easy to see that \eqref{f_S} is locally constant (but  not compactly supported) as a function of $S$.

The definition of $f_S(C)$  makes sense if we replace $C \subset V^d$ by any compact open subset in the $d$-fold product of any $F'$-Hermitian space, and it is multiplicative with respect to this data:  given $C_1 \subset V_1^d$ and $C_2\subset V_2^d$,  the product $C_1 \times C_2 \subset (V_1 \oplus V_2)^d$ satisfies 
\[
f_S(C_1 \times C_2) = f_S(C_1) f_S(C_2).
\]
In particular
\begin{equation}\label{fourier factor}
f_S(C^{[r]}) = f_S(C) \cdot f_S( \Lambda^d_{r,r} )  
=   f_S(C) \cdot f_S( \Lambda^d_{1,1} )^r .
\end{equation}

\begin{lemma}\label{lem:secret gauss}
For any  $S \in \Herm_d(F')$,  the integral
\[
f_S( \Lambda^d_{1,1} )  =  \frac{1}{ \Vol( \Lambda_{1,1}^d ) } 
\int_{ \Lambda_{1,1}^d}  \langle S ,  \tau(x) \rangle_\psi   \, dx   
\]
  is a nonpositive  power of the residue cardinality  of $F'$.
\end{lemma}
 
\begin{proof}
Write elements of $\Lambda_{1,1} = \co_{F'} \times \co_{F'}$ as pairs $(y,z)$, so that the Hermitian form on it becomes 
\[
h( (y_1,z_1) ,  (y_2,z_2) ) = y_1 \sigma(z_2) + z_1  \sigma(y_2).
\]
The moment map
\[
\co_{F'}^d \times \co_{F'}^d = \Lambda_{1,1}^d \map{\tau} \Herm_d(\co_{F'}) 
\] 
now takes the form
\[
\tau( y, z) =  \big( y_i \sigma(z_j) + z_i  \sigma(y_j)   \big)_{ i,j}  .
\]
The crucial observation is that this moment map is $\co_F$-bilinear, and satisfies
$
\tau ( \alpha y ,z) = \tau( y , \sigma(\alpha) z) 
$
for all $\alpha \in \co_{F'}$.

Normalizing the Haar measure on $\co_{F'}$ to have volume $1$, it follows easily that the inner integral in
\[
f_S( \Lambda^d_{1,1} ) 
=\int_{ \co_{F'}^d  }    \int_{ \co_{F'}^d  }   \langle S ,  \tau(y,z ) \rangle_\psi    \, dy\, dz .
\]
is  the characteristic function of some  $\co_{F'}$-submodule of $\co_{F'}^d$, completing the proof.
\end{proof}

Abbreviate $U_k = T + \pi^k \Herm_d(\co_{F'})$, so that 
\[
\frac{ \Vol(  \Herm_d( \co_{F'} )^\vee  ) } { \Vol(U_k)  } 
=  \frac{ \Vol(  \pi^{-k}  \Herm_d(\co_{F'} )^\vee   ) }{  \Vol( \Herm_d(\co_{F'}) )  } 
\]
for all $k\ge 0$.
For any choice of Haar measure on $V^d$,  the equality 
\begin{align*}
 &  \Vol ( \{ x\in  C  : \tau(x)   \in U_k  \}  )  \\
&   = 
  \int_C    \left(  \frac{1}{  \Vol( \pi^{-k}  \Herm_d(\co_{F'} )^\vee )  }  \int_{\pi^{-k} \Herm_d(\co_{F'})^\vee }   \langle S ,    \tau(x) -T  \rangle_\psi   \, dS  \right)  \, dx 
\end{align*}
holds because the quantity in parentheses is $1$ if $\tau(x)\in U_k$, and is $0$ otherwise.
Rewriting this last equality as
\begin{align*}
&   \Vol ( \{ x\in  C  : \tau(x)   \in U_k  \}  )   \\
 &  =
 \frac{ \Vol(C) }{  \Vol( \pi^{-k}  \Herm_d(\co_{F'} )^\vee )  }  \int_{\pi^{-k} \Herm_d(\co_{F'})^\vee } 
    f_S(C)   \cdot   \langle S ,     -T   \rangle_\psi         \, dS ,
\end{align*}
and comparing  with  \eqref{general density},  shows  that for all $k \gg 0$ we have \begin{align}\label{analytic density}
\Den_T(C )=
\frac{1}{  \Vol ( \Herm_d(\co_{F'}) )  } 
  \int_{\pi^{-k} \Herm_d(\co_{F'})^\vee }      f_S(C) \cdot    \langle S ,     -T  \rangle_\psi     \, dS .
\end{align}
To ease notation, we now normalize the Haar measure on $\Herm_d(F')$ to give $\Herm_d(\co_{F'} )$ volume $1$.

The proof of   \eqref{analytic density} works equally well if we replace   $C\subset V^d$  by  $C^{[r]} \subset (V^{[r]})^d$, and so it follows from \eqref{fourier factor} that
\begin{align}\label{r analytic density}
 \Den_T(C^{[r]} ) 
 =
  \int_{\pi^{-k} \Herm_d(\co_{F'})^\vee }  f_S(C) \cdot f_S( \Lambda^d_{1,1} )^r \cdot    \langle S ,     -T  \rangle_\psi     \, dS 
\end{align}
for all sufficiently large $k$.
A priori the notion of sufficiently large could depend on $r\in \Z_{\ge 0}$.
However, we may use Lemma \ref{lem:open orbit lemma} to  fix a  compact open subgroup  $K \subset \GL_d(\co_{F'})$  stabilizing $C$, and then choose $k$  large enough that    $U_k$ is contained in the $K$-orbit of $T$.
 Proposition \ref{prop:sufficiently small} then tells us that \eqref{general density} holds with $U=U_k$, and so  \eqref{analytic density} holds for this  choice of $k$.  
This particular   subgroup $K$ will also stabilize every $C^{[r]}$, and so by the same reasoning this $k$ is large enough that  \eqref{r analytic density}  holds for all  $r$ simultaneously.

 Now  consider the behavior of   \eqref{r analytic density} as a function of $r$.
 Because the integrand is locally constant as a function of $S$, and because we are integrating over a compact set, we may rewrite \eqref{r analytic density} as a finite linear combination  
\[
 \Den_T(C^{ [r] } ) = \sum_i c_i  \cdot f_{S_i}( \Lambda^d_{1,1} )^r 
\]
 for matrices $S_i \in \Herm_d(F')$ and scalars $c_i \in \C$, all independent of $r$.
Proposition \ref{prop:density polynomial} is immediate from this and Lemma \ref{lem:secret gauss}.
  \end{proof}

\begin{remark}
The density polynomial  $\Den_T(s, C)$ takes rational values whenever $s\in \Z_{\ge 0}$, and so    $\Den_T(s,C) \in \Q [ q^{-s}]$. 
\end{remark}

\begin{remark}
The construction $C \mapsto C^{[r]}$ respects disjoint unions, and so our density polynomials satisfy the obvious analogue of Remark \ref{rem:general disjoint}.
\end{remark}

%%%%%%%%%%%%%%%%%%%%%%%%%%%

\subsection{Whittaker functions}
\label{ss:whittaker}

%%%%%%%%%%%%%%%%%%%%%%%%%%%%

We quickly recall the construction of local Whittaker functions, and explain the connection with the density polynomials we have constructed.   In this subsection we assume that $\mathrm{char}(F) \neq 2$.

Fix an integer $d\ge 1$.  Abbreviating 
\[
w=\begin{pmatrix} & I_d \\  - I_d & \end{pmatrix} ,
\]  
we denote by 
\[
H_d = \left\{ g \in \GL_{2d}(F') : {}^tg  \cdot w \cdot    \sigma(g)  = w  \right\} ,
\]
the unitary group (viewed as a $p$-adic group, not an algebraic group over $F$) of the standard skew-Hermitian space of dimension $2d$. 
The \emph{standard Siegel parabolic} $P_d \subset H_d$ is defined as the stabilizer of the Lagrangian subspace $(F')^d \times \{0\} \subset (F')^{2d}$.

As in \cite[\S 12.1]{LZ1},  a character $\chi$ of $(F')^\times$ determines a  \emph{degenerate principal series} 
\[
I_d(s,\chi) = \mathrm{Ind}_{P_d}^{H_d} \big( \chi  | \cdot |_F^{ s + \frac{d}{2}} \big) 
\]
(unnormalized smooth induction).
Exactly as in \cite[\S 12.2]{LZ1},  to a standard section $\Phi( \cdot ,s) \in I_d(s,\chi)$ and a nonsingular $T \in \Herm_d(F')$, one associates a \emph{local Whittaker function} 
\[
W_T(s,  \Phi ) \in \C [q^s, q^{-s}]
\]
 of the variable $s\in \C$.   
Although we suppress it from the notation, the local Whittaker function  depends on a choice of additive character $\psi : F \to \C^\times$, which we now fix.

Let $\eta   : F^\times \to \{ \pm 1\}$ be the quadratic character determined by $F'/F$.
Fix  a finite dimensional Hermitian space $(V,h)$ for $F'/F$, and assume that  the character $\chi$ fixed above satisfies
\[
\chi|_{F^\times}= \eta^{\dim_{F'}(V)}.
\]
As recalled  in \cite[\S 12.3]{LZ1},  the choices of $\psi$ and $\chi$ determine a Weil representation of $H_d$ on the Schwartz space $S(V^d)$.  Using this one  associates to any $\varphi \in S(V^d)$ a \emph{standard Siegel-Weil section}
\[
\Phi_\varphi(\cdot ,s) \in I_d(s,\chi).
\]
By mild abuse of notation, we denote by $W_T(s,\varphi) = W_T(s,\Phi_\varphi)$ the associated local Whittaker function.

In the following proposition, we assume that the additive character $\psi$ is unramified.
Take the  Haar measure on $\Herm_d(F')$ to be  self-dual with respect to the pairing $(a,b) \mapsto \psi( \mathrm{Tr} ( ab) )$, take the Haar measure on $V$ to be self-dual with respect to  $(x,y)\mapsto \psi ( \mathrm{Tr}_{F'/F} h(x,y) )$, and  understand the volume of $C$ to be taken with respect to the product measure on $V^d$.

\begin{proposition}\label{prop:whittaker}
Let   $\bm{1}_C\in S(V^d)$ be the characteristic function of a compact open subset $C \subset V^d$.
For any nonsingular $T \in \Herm_d(F')$ we have 
\[
W_T (s, \bm{1}_C )
= \gamma (V)^d \cdot     \Vol(  C)   \cdot  \mathrm{Vol}(  \Herm_d(\co_{F'}) ) \cdot   \Den_T\big( s, C \big)  .
\]
Here  the Weil constant $\gamma(V)$ appearing on the right-hand side  is the eighth root of unity defined by \cite[(10.3)]{KR2}. 
\end{proposition}

\begin{proof}
The proof is  the same as that of \cite[Proposition 10.1]{KR2}, which treats the case in which $C=M^d$ for a lattice $M \subset V$ on which the Hermitian form is $\co_{F'}$-valued.

Briefly, because both sides of the desired equality lie in $\C[ q^s,q^{-s} ]$, it suffices to prove the equality when $s=r \in \Z_{\ge 0}$.   By the interpolation trick of Rallis, we have
\[
W_T(r,\bm{1}_C) = 
 \gamma(V)^d  
\int_{\Herm_d(F') }  \psi( -\mathrm{Tr}(bT) )
\left(   \int_{  C^{[r]}   }   \psi ( \mathrm{Tr}( b \tau(x) )  )    \, dx   \right)   \, db,
\]
where  the Haar measure on  $( V^{[r]})^d$ is normalized in the same way as that  on $V^d$.
Fixing a uniformizer $\pi \in \co_F$, we may  compute the outer integral as the limit as $k\to \infty$ of the integrals over $ \pi^{-k} \Herm_d(\co_{F'})$.  In this way we see that 
\begin{align*}
W_T(r, \bm{1}_C ) &=  \gamma(V)^d  \lim_{k\to \infty} 
 \int_{  C^{[r]}   }
\left(    \int_{\pi^{-k} \Herm_d(\co_{F'}) }  \psi ( \mathrm{Tr}( b ( \tau(x) -T) )  )    \, db   \right)   \, dx  \\
& =
\gamma(V)^d  \lim_{k\to \infty} 
 \int_{  C^{[r]}    }
 \mathbf{1}_{ U_k} (\tau (x) ) 
\left(    \int_{\pi^{-k} \Herm_d(\co_{F'})}  1   \, db   \right)   \, dx  ,
\end{align*}
where in the final line we have set   $U_k = T + \pi^k \Herm_d(\co_{F'} )^\vee$.
The inner integral is 
\[
 \int_{\pi^{-k} \Herm_d(\co_{F'})}  1   \, db 
 = 
  \mathrm{Vol}(  \Herm_d(\co_{F'})   )  
 \frac{   \mathrm{Vol}(  \Herm_d(\co_{F'})^\vee   )    }{   \mathrm{Vol}( U_k  )   } , 
\]
and using this the claim follows from Proposition \ref{prop:natural density}.
\end{proof}

%%%%%%%%%%%%%%%%%%%%%%%%%%%%%%%

\section{Tools for computation}
\label{s:tools}

%%%%%%%%%%%%%%%%%%%%%%%%%%%%%%%

In this section we develop two technical tools for the explicit calculation of representation densities.
The first allows one to compute representation densities using arithmetic geometry, by counting points on the reductions of well-chosen $\co_F$-schemes.  The second is a reduction formula that allows one to compute the representation densities of block diagonal matrices in terms of representation densities of the diagonal blocks.  These are the  key methods needed for the calculations in \S \ref{s:density formulas}.

We continue to work with a finite dimensional vector space $V$ over $F'$, endowed with a nondegenerate Hermitian form $h$.
Fix an integer $d\ge 1$.

%%%%%%%%%%%%%%%%%%%%%%%%%%%%%%%

\subsection{Integral models}

%%%%%%%%%%%%%%%%%%%%%%%%%%%%%%%

We now explain how representation densities can be computed using arithmetic geometry, generalizing the ideas used in \cite{GY} and \cite{CY} for the classical representation densities of Remark \ref{rem:classical density}.

Suppose we are given smooth $\co_F$-schemes $\mathcal{X}$ and $\mathcal{Y}$ whose generic fibers are the $F$-schemes $X$ and $Y$ defined by \eqref{XandY}.  In particular, such integral models determine compact open subsets
\[
\mathcal{X}(\co_F) \subset X(F)=V^d \qquad \mbox{and} \qquad \mathcal{Y}(\co_F) \subset  Y(F)=\Herm_d(F').
\]
Assume  that the moment morphism \eqref{moment morphism} extends (necessarily uniquely) to a morphism
\[
\tau : \mathcal{X}   \to \mathcal{Y},
\]
and that  the  translation-invariant forms  $\mu_{X}$ and $\mu_{Y}$ from \eqref{mewtwo} can be extended  (necessarily uniquely) to nowhere vanishing  forms
\[
\mu_{\mathcal{X}}  \in H^0(  \mathcal{X}  , \Omega_{\mathcal{X}}^{\dim(X)} ) 
\qquad\mbox{and}\qquad
\mu_{\mathcal{Y}}  \in H^0( \mathcal{Y} , \Omega_{\mathcal{Y}}^{\dim(Y)} ) .
\]

Let $\mathcal{U} \subset \mathcal{X}$ be a Zariski open subset on which the morphism $\tau$ is smooth, 
and for any nonsingular Hermitian matrix   $T\in \mathcal{Y}(\co_F)$, define a smooth  $\co_F$-scheme $\mathcal{U}_T$ by the cartesian diagram
\begin{equation}\label{integral cartesian}
\begin{tikzcd}
{ \mathcal{U}_T }\ar[r] \ar[d]  & { \Spec(\co_F) } \ar[d , "T" ]  \\
{ \mathcal{U} } \ar[r , "\tau" '] & { \mathcal{Y} } .
\end{tikzcd}
\end{equation}
In particular,  $\mathcal{U}_{T} (\co_F) \subset X_T(F)$ is a compact open subset of \eqref{X_T points}.

The following result shows that  computing representation densities can be reduced to the problem of counting points in the special fibers of smooth integral models.

\begin{proposition}\label{prop:integral densities}
If $C\subset X(F)$ is a compact open subset satisfying
\[
C \cap X_T(F)  = \mathcal{U}_T(\co_F),
\]
 then the representation density of $T$ by $C$ is 
\[
\Den_T(C) 
 =  \#  \mathcal{U}_T( k )  
\cdot
\frac{ \# \mathcal{Y}( k )   }{  \# \mathcal{X}(k)   }
\cdot \frac{   \Vol( \mathcal{X}(\co_F) )   }{    \Vol(C)  }
\cdot \frac{ \Vol(  \Herm_d(\co_{F'})^\vee  ) }{ \Vol(  \mathcal{Y}(\co_F)  ) } .
\]
Here $k=\co_F/\pi\co_F$ is the residue field of $\co_F$.
\end{proposition}

\begin{proof}
Fix a point $x\in \mathcal{X}(k)$, and let  $V_x \subset \mathcal{X}(\co_F)$ be the set of points reducing to $x$ in the special fiber.
 As $\mathcal{X}$ is smooth over $\co_F$, the completed local ring at $x$ is isomorphic to a power series ring, say
\[
\widehat{\co}_{\mathcal{X},x}  \iso \co_F[[ t_1,\ldots, t_{\dim(X)}]] , 
\]
and  these coordinates determine an  isomorphism 
$
V_x \iso  ( \pi \co_F)^{\dim (X) } 
$
of $F$-analytic manifolds.  Under this isomorphism, 
$
\mu_X = f \cdot   dt_1 \wedge \cdots  \wedge d t_{\dim(X)} 
$
for some $f$ in the power series ring, and the assumption that $\mu_X$ admits a nowhere vanishing extension  to $\mathcal{X}$ implies that $f$ is a unit.
In particular, $| f |_F =1$ as analytic functions on $( \pi \co_F)^{\dim (X) }$.
Recalling the definition of $F$-analytic integration from \cite[Chapter 7.4]{Igusa}, we have
\[
\mu_X (V_x) = \int_{  ( \pi \co_F)^{\dim (X) }  }   | f |_F \cdot dt_1 \wedge \cdots  \wedge d t_{\dim(X)}  =
q^{ - \dim(X) }.
\]

By varying the point $x$, we find that 
\begin{equation}\label{integral volume 1}
\mu_X( \mathcal{X}(\co_F) )  = \# \mathcal{X}( k ) \cdot q^{ - \dim(X) }. 
\end{equation}
Similar reasoning shows that 
\begin{equation}\label{integral volume 2}
\mu_Y  ( \mathcal{Y}(\co_F) )  = \# \mathcal{Y}(k) \cdot q^{ - \dim(Y) }.
\end{equation}

Using the assumption that $\tau : \mathcal{U} \to \mathcal{Y}$ is smooth, 
the restriction to $\mathcal{U}_{T/F} \subset X_T$ of the top-degree form $\mu_{X_T}$ of  \eqref{mewT}  extends (by repeating the construction of \eqref{mewT} on the level of integral models) to a nowhere vanishing form on the smooth $\co_F$-scheme  $\mathcal{U}_T$.  
With this in hand,  the same reasoning as above proves
\begin{equation}\label{integral volume 3}
\mu_{X_T} ( \mathcal{U}_T(\co_F) ) =  \#\mathcal{U}_T(k ) \cdot q^{ - \dim( X_T)} .
\end{equation}

The claim follows by rewriting Definition \ref{def:density} as
\begin{align*}
\Den_T(C) 
%& =   \mu_{ X_T  }(    \mathcal{U}_T( \co_F)  ) 
%\cdot \frac{ \mu_Y(  \Herm_d(\co_{F'})^\vee   ) }{ \mu_X(C) } \\
& =   \mu_{ X_T  }(    \mathcal{U}_T( \co_F)  ) 
\cdot \frac{ \mu_Y( \mathcal{Y}(\co_F)  ) }{ \mu_X( \mathcal{X}(\co_F)) }  
\cdot \frac{   \Vol( \mathcal{X}(\co_F) )   }{    \Vol(C)  }
\cdot \frac{ \Vol(  \Herm_d(\co_{F'})^\vee  ) }{ \Vol(  \mathcal{Y}(\co_F)  ) }
\end{align*}
and using the  formulas \eqref{integral volume 1}, \eqref{integral volume 2}, and \eqref{integral volume 3}.
\end{proof}

%%%%%%%%%%%%%%%%%%%%%%%%%%%%%%%

\subsection{A reduction formula}
%\label{ss:reduction}

%%%%%%%%%%%%%%%%%%%%%%%%%%%%%%%

Fix a decomposition  $d= d_1 + d_2$ with $d_1$ and $d_2$ positive integers.
Identifying  $V^d = V^{d_1} \times V^{d_2}$ in the obvious way, there is a corresponding decomposition
\begin{equation}\label{Xsplit}
X = X_1 \times X_2 ,
\end{equation}
where $X_i = V^{d_i}$ regarded as an  $F$-scheme as in  \eqref{XandY}. 

Fix  a block diagonal Hermitian matrix 
\[
T  = \begin{pmatrix}  T_1 & \\ & T_2  \end{pmatrix} \in \Herm_d (F')
\]
with each $T_i \in \Herm_{d_i}( F')$  nonsingular.
There  are corresponding closed subschemes 
\[
X_T \subset X, \qquad X_{T_1} \subset X_1,  \qquad X_{T_2} \subset X_2
\]
parametrizing tuples of vectors in $V$ with moment matrices $T$, $T_1$, and $T_2$, respectively,  exactly as in \eqref{X_T points}.

Given a compact open subset $C \subset X(F)$, we have defined the representation density $\Den_T(C)$ in terms of the  volume $\mu_{X_T}( C \cap X_T(F))$, which of course can be interpreted as an integral over $X_T(F)$.    
Projection to the second factor in $V^d = V^{d_1} \times V^{d_2}$ restricts to a morphism of smooth $F$-schemes 
\begin{equation}\label{fiber morphism}
\pi_2 : X_T \to X_{T_2} ,
\end{equation}
sending a $d$-tuple of vectors with moment matrix $T$ to its last $d_2$-components, which of course have moment matrix $T_2$.

There is a Fubini-style formula (as in Proposition \ref{prop:fubini})  expressing integrals on $X_T$ as iterated integrals, first along the fibers of \eqref{fiber morphism}, then along the base. 
Our goal in this subsection is to rewrite this  formula for fiberwise integration  in such a way that it expresses   $\Den_T(C)$ in terms of representation densities for the diagonal blocks $T_1$ and $T_2$. 
To do this, we  must first carefully specify the normalizations of the many different measures that will be involved.

 Using the factorization \eqref{Xsplit}, we identify   
 \[
 \Omega^{\dim(X)}_X \iso 
 \Omega^{\dim(X_1)}_{X_1} \boxtimes \Omega^{\dim(X_2)}_{X_2} 
 \]
 (the right-hand side means the tensor product as $\co_X$-modules of the pullbacks of  the two factors). 
As in \eqref{mewtwo},  choose nonzero translation-invariant top-degree forms $\mu_{X_1}$ and $\mu_{X_2}$ on $X_1$ and $X_2$.  
These determine a top-degree form 
\begin{equation*}
\mu_X=    \mu_{X_1}  \boxtimes   \mu_{X_2} \in H^0 (X , \Omega^{\dim(X)}_X ) .
\end{equation*}
The $F$-scheme $Y$ of $d\times d$ Hermitian matrices  decomposes as
\begin{equation}\label{Ysplit}
Y = Y_1 \times Y_2 \times M,
\end{equation}
where we have set 
\begin{align*}
Y_i   = \Herm_{d_i}(F')  \quad \mbox{and} \quad 
M   = \mathrm{Mat}_{d_1\times d_2}(F'),
\end{align*}
but regarded as $F$-schemes as in \eqref{XandY}.
The isomorphism is
\[
\begin{pmatrix}  T_1 & S \\ {}^t S^\sigma & T_2 \end{pmatrix} \mapsto (T_1, T_2, S) .
\]
As with  $\mu_X$, we fix a decomposition 
\begin{equation*}
\mu_Y = \mu_{Y_1}\boxtimes \mu_{Y_2} \boxtimes \mu_M \in H^0(Y , \Omega_Y^{\dim(Y)} ) 
\end{equation*}
of the translation-invariant form \eqref{mewtwo}.
Exactly as in \eqref{mewT}, the above choices determine  nonvanishing top-degree forms 
\[
\mu_{X_{T_i}} \in  H^0 \big( X_{T_i} , \Omega_{X_{T_i}}^{ \dim(X_{T_i} ) }   \big) ,
\]
and the induced measures on the $F$-analytic manifolds $X_{T_i}(F)$ are denoted in the same way.

For any $z = (z_1,\ldots, z_{d_2})  \in V^{d_2}$ with moment matrix $T_2$,  let $V(z) \subset V$ be the subspace spanned by the vectors $z_1,\ldots, z_{d_2}$, and let $V(z^\perp) \subset V$ be the subspace of vectors orthogonal to all  $z_1,\ldots, z_{d_2}$.  Because $T_2$ is nonsingular, $V(z)$ is a nondegenerate Hermitian space of dimension $d_2$, and there is an orthogonal decomposition 
\[
V = V(z^\perp) \oplus V( z ).
\]
This allows us to  identify
\[
V^{d_1} = V(z^\perp)^{d_1} \times V(z)^{d_1}.
\]
On the level of $F$-schemes, there is a corresponding   decomposition
\begin{equation}\label{z splitting}
X_1 = X_1(z^\perp) \times X_1(z),
\end{equation}
in which the first factor  parametrizes $d_1$-tuples of vectors in $V(z^\perp)$, and the second parametrizes $d_1$-tuples of vectors in $V(z)$.

Again using the assumption that $\tau(z) = T_2$ is nonsingular,  there is  a short exact sequence
\begin{equation}\label{h_z def}
0 \to V(z^\perp)^{d_1} \to V^{d_1} \map{ h_z }  M(F) \to 0 ,
\end{equation}
in which $h_z$ is the $F'$-linear map  sending a tuple $(y_1,\ldots, y_{d_1}) \in V^{d_1}$ to the $d_1 \times d_2$ matrix of Hermitian  pairings $h_z(y) = ( h(y_i,z_j))$.  This map   restricts to an $F'$-linear isomorphism 
\begin{equation}\label{h_z iso}
h_z:V(z)^{d_1} \iso  M(F) ,
\end{equation}
and so the chosen top-degree form $\mu_M$ on $M$ pulls back to a top-degree form  
\[
\mu_{X_1(z)} = h_z^* \mu_M
\]
 on $X_1(z)$.  We then normalize the translation-invariant  top-degree form $\mu_{X_1(z^\perp)} $ on $X_1(z^\perp)$ in such a way that,  under the decomposition \eqref{z splitting}, 
 \[
\mu_{X_1} = \mu_{X_1(z^\perp)}   \boxtimes \mu_{X_1(z)} .
\]
In other words, the measure  $\mu_{X_1(z^\perp)} $ on $V(z^\perp)^{d_1}$ is normalized so that 
\begin{equation}\label{X_1(z) quotient measure}
\mu_{X_1(z^\perp)}  (  A  )  = \frac{  \mu_{X_1}( A \times B)   }{  \mu_M( h_z(B)  )   }  
\end{equation}
for any nonempty compact open subsets $A \subset V(z^\perp)^{d_1}$ and $B \subset V(z)^{d_1}$.

\begin{theorem}\label{thm:reduction formula}
Fix a nonempty compact open subset $C \subset V^d$.  For  any $z \in V^{d_2}$ with moment matrix $T_2$,  define a (possibly empty) compact open subset
\begin{align*}
C_{z^\perp}  & = \{ y \in V(z^\perp)^{d_1} : (y,z) \in C \}  \subset X_1(z^\perp)(F). 
\end{align*}
Here $(y,z) \in V^d$ is the concatenation of the $d_1$-tuple $y\in V^{d_1}$ with the $d_2$-tuple $z\in V^{d_2}$.
The representation density of $T$ with respect to $C$ satisfies
\begin{align*}
 \Den_T(C)    & = 
 \frac{1}{ \mu_X(C) } \cdot 
 \frac{ \mu_Y( \Herm_d(\co_{F'})^\vee )  } {  \mu_{Y_1}(  \Herm_{d_1}(\co_{F'})^\vee  )  }   \\
 & \quad \times
 \int_{ X_{T_2}(F) } 
  \Den_{T_1}(C_{z^\perp}) \cdot \mu_{X_1 (z^\perp) } ( C_{z^\perp})  
 \,  d \mu_{X_{T_2}}  (z) .
\end{align*}
Inside the integral,  the representation density of $T_1$ with respect to $C_{z^\perp}$ is computed inside the subspace  $V(z^\perp) \subset V$, with its restricted Hermitian form. 
\end{theorem}

Theorem \ref{thm:reduction formula} generalizes known reduction formulas for  classical representation densities, such as \cite[Theorem 5.6.2]{Kitaoka}. 
Our proof, which  will occupy the remainder of this subsection,  is different but equally tedious.

Suppose we are given a point $z \in X_{T_2}(F)$.  In other words, 
a  tuple  $z  \in  V^{d_2}$ with moment matrix $T_2$.
Consider the   closed subscheme
\[
X_{T_1}( z^\perp)  =  X_{T_1} \cap X_1(z^\perp) \subset X_1
\]
of \eqref{z splitting},  parametrizing those    $ y \in V(z^\perp)^{d_1}$ with moment matrix $T_1$.
It admits  a closed immersion
 \begin{equation}\label{zinc}
 X_{T_1}( z^\perp)  \map{ y \mapsto (y,z) } X_T 
 \end{equation}
 sending $y$ to the concatenated tuple $(y,z) \in V^d$ with moment matrix $T$, and the image of \eqref{zinc}  is the fiber of \eqref{fiber morphism} above $z$.

Consider the restriction 
$
\tau_1 : X_1(z^\perp) \to Y_1 
$
of the  moment morphism $\tau_1 : X_1 \to Y_1$ to the closed subscheme $X_1(z^\perp) \subset X_1$.
As in \eqref{smooth moment}, it restricts to a morphism
\[
\tau_1 : X_1(z^\perp)^\mathrm{ns} \to Y_1^\mathrm{ns}, 
\]
where $Y_1^\mathrm{ns} \subset Y_1$ is the open subscheme parametrizing nonsingular Hermitian matrices, and $X_1(z^\perp)^\mathrm{ns} \subset   X_1(z^\perp)$ is its preimage under $\tau_1$.  
Denote by $\mu_{X_1(z^\perp)^\mathrm{ns}}$ and $\mu_{Y_1^\mathrm{ns}}$ the restrictions of the already defined forms $\mu_{X_1(z^\perp)}$ and $\mu_{Y_1}$ to these open subschemes.

\begin{lemma}\label{lem:fiber integral 1}
For any point $z \in X_{T_2}(F)$,  we  have a commutative diagram
 \begin{equation}\label{factor double fiber}
\begin{tikzcd}
{X_T} \ar[d , "\pi_2" ' ]  &  { X_{T_1}( z^\perp )  }   \ar[l , " \eqref{zinc} " '  ]  \ar[r] \ar[d] & { X_1(z^\perp)^\mathrm{ns} } \ar[d , " \tau_1" ]   \\
 {   X_{T_2} } &  { \Spec(F) }  \ar[l , "z "]   \ar[r , " T_1" ' ]   & {Y_1^\mathrm{ns} } 
\end{tikzcd}
\end{equation}
in which both squares are cartesian,  and the vertical morphisms  are smooth.  
\end{lemma}

\begin{proof}
The cartesian claim is clear after unwinding  the definitions.

For the smoothness claims, let us ease notation by replacing $F$ by an algebraic closure, and $F'$ by $F\times F$.
Given a point 
$
y\in X_1(z^\perp)^{\mathrm{ns}}(F),
$
the nonsingularity of the moment matrix  $\tau_1(y)$ implies that the components of $y$ generate a free $F'$-submodule of $V(z^\perp)$ of rank $d_1$.  
It now follows from Lemma \ref{lem:tangent moment}, applied to the Hermitian space
$V(z^\perp)$, that the arrow  labeled $\tau_1$ in \eqref{factor double fiber} is smooth.

The middle vertical arrow in \eqref{factor double fiber} is smooth because it is the base change of a smooth morphism.
In particular, this proves that  $X_{T_1}(z^\perp)$ is a smooth $F$-scheme of dimension 
\begin{align}\label{X_T dimension}
\dim X_{T_1}(z^\perp) & = 
 \dim X_1(z^\perp)^\mathrm{ns}   - \dim Y_1^\mathrm{ns}   \\
 & = d_1\bigl(  \dim_{F}V - 2 d_2\bigr)-  d_1^2  \nonumber  \\
 & = d (\dim_F V - d) - d_2( \dim_F V -d_2 ) \nonumber  \\
& = \dim X_T-\dim X_{T_2}.  \nonumber 
\end{align}

Using the fact that the square on the left in \eqref{factor double fiber} is cartesian, we may allow $z$ to vary and deduce from the previous paragraph that  the arrow labeled $\pi_2$ in \eqref{factor double fiber} has smooth fibers all of dimension $\dim X_T-\dim X_{T_2}$.  Using the  miracle flatness criterion from  
\cite[\href{https://stacks.math.columbia.edu/tag/00R4}{Lemma 00R4}]{stacks-project}, we find that $\pi_2$
is a flat morphism between smooth $F$-schemes with smooth fibers, and hence is smooth.
\end{proof}

Given tuples $y \in V^{d_1}$ and $z \in V^{d_2}$, let 
\[
h_{12}(y,z)= \big(h(y_i,z_j)\big)_{i,j}    \in \mathrm{Mat}_{d_1 \times d_2}(F')
\]
 be the upper-right  block of the moment matrix of $(y,z) \in V^d$.   We regard this as a morphism of $F$-schemes
\[
h_{12}:X_1\times X_2\longrightarrow M .
\]
 Unpacking the definitions, one obtains a cartesian diagram
 \begin{equation}\label{Phi diagram}
\begin{tikzcd}
{  X_T }  \ar[rr ]\ar[d]&  &  { X^\mathrm{ns}_1\times X_{T_2} } \ar[d,"\Phi"] \\
{ X_{T_2} }  \ar[rr," { z \mapsto ( T_1 ,0 , z )}" '] & &   {    Y^\mathrm{ns}_1\times M   \times X_{T_2} } 
\end{tikzcd}
\end{equation}
in which the right vertical arrow is 
 \[
 \Phi(y,z) = \left( \tau_1(y),h_{12}(y,z),z \right) .
 \]
Taking the fiber of the entire diagram at a fixed $z \in X_{T_2}(F)$ results in the cartesian diagram
\begin{equation}\label{mini z fiber}
\begin{tikzcd}
{ X_{T_1}(z^\perp) }    \ar[r] \ar[d]  
&  {  X_1^\mathrm{ns} } \ar[d, "{(\tau_1,h_z) }" ]  \\
{ \Spec(F) } \ar[r , " {(T_1,0)}" ']   &  {  Y_1^\mathrm{ns}  \times M   } ,
\end{tikzcd}
\end{equation}
where $h_z : X_1 \to M$ is the morphism from \eqref{h_z def}.

\begin{lemma}\label{lem:Phi smooth}
The morphism $\Phi$ is smooth at every point of $X_T$.
In particular, the right vertical arrow in \eqref{mini z fiber} is smooth at every point of $X_{T_1}(z^\perp)$.
\end{lemma}

\begin{proof}
Because smoothness can be checked after extending scalars from $F$ to an algebraic closure, we ease notation by assuming  $F$ is algebraically closed, and $F'=F\times F$.
Fix a geometric point
\[
(y,z)\in  X_T(F),
\]
so that $y\in X_1(F)$, $z \in X_2(F)$, and the  concatenated tuple $(y,z) \in X$ has moment matrix  $T$. 
The geometric fiber of $\Phi$ over 
\[
\Phi(y,z) = ( T_1 , 0 , z ) 
\]
 is naturally identified with $X_{T_1}(z^\perp)$, which  is a smooth $F$-scheme of dimension \eqref{X_T dimension}.
 One can easily check that this dimension agrees with 
\[
\dim\big(X_1^{\mathrm{ns}}\times X_{T_2}\big)
-\dim\big( Y_1^{\mathrm{ns}}\times M \times X_{T_2}\big) 
= \dim ( X_1^{\mathrm{ns}}  )
-\dim\big( Y_1^{\mathrm{ns}}\times M \big) .
\]
In other words, the fiber of $\Phi$ over $\Phi(y,z)$ is smooth of  the expected dimension, and because  the source and target of $\Phi$ are smooth, it follows that $\Phi$ is smooth at $(y,z)$.  As in the proof of Lemma \ref{lem:fiber integral 1}, this uses  \cite[\href{https://stacks.math.columbia.edu/tag/00R4}{Lemma 00R4}]{stacks-project} for  the  flatness of $\Phi$ at $(y,z)$.
\end{proof}

Recall  the constructions from  \eqref{quotient form} and \eqref{mewT}: Suppose we are given a cartesian diagram
\[
\begin{tikzcd}
{ P_s  }  \ar[r]   \ar[d] &  {  P  } \ar[d , " f " ] \\
{ \Spec(F) } \ar[r, " s " ] & { S }  
\end{tikzcd}
\]
in which $f$   is a smooth morphism between smooth equidimensional $F$-schemes.
If  $\mu_P$ and $\mu_S$ are nowhere vanishing top-degree forms on $P$ and $S$, there is a  top-degree relative form 
\[
\frac{\mu_P}{ f^*\mu_S} \in H^0( P , \Omega_{P/S}^{ \mathrm{top}} ) 
\]
  characterized by the equality
\[
\mu_P
=
\frac{\mu_P}{f^*\mu_S} \otimes f^*\mu_S
\]
under the canonical isomorphism
$
\Omega_P^{\mathrm{top}} \iso \Omega_{P/S}^{ \mathrm{top} } \otimes f^* \Omega_S^{\mathrm{top}} .
$
It restricts to a  top-degree form on the fiber, denoted
\begin{equation}\label{general fiber form}
\frac{\mu_P}{f^*\mu_S}  \Big|_{P_s} \in H^0( P_s , \Omega^\mathrm{top}_{P_s} ) .
\end{equation}
Note that in order to construct \eqref{general fiber form}  one only needs $f$ to be smooth at every point of (hence in an open neighborhood of) the fiber $P_s$, not at every point of $P$.

For any point $z \in X_{T_2}(F)$,   we may  apply  this construction  to the two cartesian squares in  \eqref{factor double fiber}.  This results in  two top-degree forms on  the common fiber $X_{T_1}(z^\perp)$, and the  technical core of the proof of Theorem \ref{thm:reduction formula} is the relation 
\begin{equation}\label{the two forms}
(-1)^{d_1 d_2}  \cdot 
\frac{\mu_{X_T} }{  \pi_2^*\mu_{ X_{T_2} } }  \Big|_{X_{T_1}( z^\perp) }  
=
\frac{ \mu_{X_1(z^\perp)^\mathrm{ns} }}{ \tau_1^* \mu_{Y_1^\mathrm{ns}}  }  \Big|_{X_{T_1}( z^\perp) }  
\end{equation}
between them. 
We will prove  \eqref{the two forms} by showing that both sides are equal to the  top-degree form
\begin{equation}\label{auxiliary fiber form}
\nu_z
=
\frac{\mu_{X_1^{\mathrm{ns}}}}
     {      ( \tau_1^* \mu_{Y_1^{\mathrm{ns}}} )  \wedge ( h_z^* \mu_M)} \Big|_{X_{T_1}(z^\perp)} 
\end{equation}
associated to the diagram \eqref{mini z fiber}.

\begin{lemma}\label{lem:second nu}
The form \eqref{auxiliary fiber form} satisfies
\begin{align*}
\nu_z
&=
\frac{
\mu_{X_1(z^\perp)^{\mathrm{ns}}}
}{
\tau_1^*\mu_{Y_1^{\mathrm{ns}}}
}
\Big|_{X_{T_1}(z^\perp)}.
\end{align*}
\end{lemma}

\begin{proof}
We use the  decomposition  $X_1 = X_1(z^\perp)\times X_1(z)$  from  \eqref{z splitting}.
 Under this identification, the morphism $(\tau_1,h_z)$ in \eqref{mini z fiber} is equal to the morphism
\[
X_1= X_1(z^\perp)\times X_1(z) \map{ F_z}  Y_1\times M,
\]
defined by 
\[
F_z(y,y')
=
\bigl(\tau_1(y)+\tau_1(y'),h_z(y')\bigr).
\]
Here we have used the orthogonality of all components of $y\in V(z^\perp)^{d_1}$ to all components of  $y'\in V(z)^{d_1}$ to obtain
\[
\tau_1(y+y')=\tau_1(y)+\tau_1(y'),
\]
and the orthogonality of all components of $y\in V(z^\perp)^{d_1}$ to all components of $z\in V^{d_2}$ to obtain  $h_z(y+y') = h_z(y')$.
The fiber of $F_z$ over $(T_1,0)$ is
\[
 X_{T_1}(z^\perp)\times\{0\} \subset X_1.
\]

Consider also the morphism
\[
X_1= X_1(z^\perp)\times X_1(z) \map{G_z}  Y_1\times M
\]
defined by
\[
G_z(y,y')
=
\bigl(\tau_1(y),h_z(y')\bigr).
\]

The maps $F_z$ and $G_z$  agree on  $X_{T_1}(z^\perp)\times\{0\}$, and induce the same map on tangent spaces at every point of that closed subscheme.  
To see why, use the $F$-vector space structure on $Y_1 \times M$ to form the difference 
\[
( F_z - G_z)( y , y') = ( \tau_1(y') , 0)  .
\]
Because the moment map $ \tau_1 : V^{d_1} \to \Herm_{d_1}(F')$ is a homogeneous map of degree two  between $F$-vector spaces, both it and its differential vanish at the origin.  Hence the same is true of $F_z - G_z$ at any point of the form $(y,0)$.

The quotient form \eqref{general fiber form} on a smooth fiber is determined by the induced
cotangent map along that fiber.  Therefore $F_z$ and $G_z$ determine the
same quotient form 
\[
\frac{ \mu_{X_1^\mathrm{ns}} }{  F_z^*( \mu_{Y_1^\mathrm{ns}} \boxtimes \mu_M) } \Big|_{ X_{T_1}(z^\perp)  } 
= 
\frac{ \mu_{X_1^\mathrm{ns}} }{  G_z^*( \mu_{Y_1^\mathrm{ns}} \boxtimes \mu_M) } \Big|_{ X_{T_1}(z^\perp)  } 
\]
on the fiber $X_{T_1}(z^\perp)\times\{0\} \subset X_{T_1}$.

We normalized the top-degree form $\mu_{X_1(z^\perp)}$ in such a way that, under the decomposition \eqref{z splitting},   we have
\[
\mu_{X_1}
=
\mu_{X_1(z^\perp)}
\boxtimes h_z^*\mu_M.
\]
This allows us  to compute
\begin{align*}
\nu_z
&=
\frac{\mu_{X_1^{\mathrm{ns}}}}    {F_z^*      (\mu_{Y_1^{\mathrm{ns}}}\boxtimes\mu_M)} \Big|_{X_{T_1}(z^\perp)  \times \{0\}}  \\
&=
\frac  {\mu_{X_1^{\mathrm{ns}}}}    {G_z^*(\mu_{Y_1^{\mathrm{ns}}}\boxtimes\mu_M)
}
\Big|_{X_{T_1}(z^\perp)\times\{0\}}\\
&=
\frac{
\mu_{X_1(z^\perp)^{\mathrm{ns}}}
\boxtimes h_z^*\mu_M
}{
\tau_1^*\mu_{Y_1^{\mathrm{ns}}}
\boxtimes h_z^*\mu_M
}
\Big|_{X_{T_1}(z^\perp)\times\{0\} }\\
&=
\frac{
\mu_{X_1(z^\perp)^{\mathrm{ns}}}
}{
\tau_1^*\mu_{Y_1^{\mathrm{ns}}}
}
\Big|_{X_{T_1}(z^\perp)}
\end{align*}
as claimed.
\end{proof}

\begin{lemma}\label{lem:first nu}
We have the equality
\[
(-1)^{d_1d_2}   \nu_z  =   \frac{\mu_{X_T}}  {\pi_2^*\mu_{X_{T_2}}}  \Big|_{X_{T_1}(z^\perp)}.
\]
\end{lemma}

\begin{proof}
It is enough to prove the identity locally on  $X_{T_1}(z^\perp)$.  
Fix a point $y$ of this scheme, identify it with its image   $(y,z)\in X_T$ under \eqref{zinc}, and work on an open neighborhood of that point in $X_1^{\mathrm{ns}}\times X_{T_2}$ small enough that the morphism 
\[
\Phi:
X_1^{\mathrm{ns}}\times X_{T_2}
\longrightarrow
  Y_1^{\mathrm{ns}}\times M  \times X_{T_2}
\]
of \eqref{Phi diagram} is smooth (Lemma \ref{lem:Phi smooth}). 
By smoothness of $\Phi$, and after further shrinking our neighborhood of $(y,z)$  if necessary, there
is a differential form $\alpha$ on $X_1^{\mathrm{ns}}\times X_{T_2}$ such that
\begin{align}\label{alpha factorization for Phi}
\mu_{X_1^{\mathrm{ns}}}\boxtimes\mu_{X_{T_2}}
& =
\alpha  \wedge \Phi ^*  \big(\mu_{Y_1^{\mathrm{ns}}}\boxtimes\mu_M  \boxtimes \mu_{X_{T_2}}  \big)   .  
\end{align}

For a fixed $z \in X_{T_2}(F)$, the diagram  \eqref{mini z fiber} can be extended to
\[
\begin{tikzcd}
{ X_{T_1}(z^\perp) }    \ar[r]   \ar[d]   & {  X^\mathrm{ns}_1 }  \ar[rrr , " {y \mapsto (y,z) }"  ]\ar[d , "  {  ( \tau_1 , h_z )  } "  ]&  &  &  { X^\mathrm{ns}_1\times X_{T_2} } \ar[d,"\Phi"] \\
{ \Spec(F) }  \ar[r, " { (  T_1,0  )   }  " ']  &  { Y^\mathrm{ns}_1\times M }  \ar[rrr  , " { ( T_1 , A) \mapsto (T_1,A, z)   } " ' ] & &   & {    Y^\mathrm{ns}_1\times M   \times X_{T_2} } 
\end{tikzcd}
\]
with both squares cartesian.  
Directly from the definition    \eqref{auxiliary fiber form},  one finds the equalities
\begin{equation}\label{mid nu_z}
\nu_z  
   = \frac{ \mu_{X_1^{\mathrm{ns}}}\boxtimes\mu_{X_{T_2}} }{  \Phi^* \big(\mu_{Y_1^{\mathrm{ns}}}\boxtimes\mu_M \boxtimes  \mu_{X_{T_2}} \big)     }\Big|_{X_{T_1}(z^\perp)     }     =      \alpha|_{X_{T_1}(z^\perp)  }
\end{equation}
of forms on $X_{T_1}(z^\perp) = X_{T_1}(z^\perp) \times \{ z\} \subset X_1^\mathrm{ns} \times X_{T_2}$.

On the other hand, we may similarly extend \eqref{Phi diagram} to
\[
\begin{tikzcd}
{ X_{T_1}(z^\perp) }    \ar[r]   \ar[d]   &   {  X_T }  \ar[rr ]\ar[d , " \pi_2" ] &  &  { X^\mathrm{ns}_1\times X_{T_2} } \ar[d,"\Phi"] \\
{ \Spec(F) }  \ar[r, " z " ']  &   { X_{T_2} }  \ar[rr," { z' \mapsto ( T_1 ,0 , z' )}" '] & &   {    Y^\mathrm{ns}_1\times M   \times X_{T_2} .} 
\end{tikzcd}
\]
We claim that there is an equality 
\begin{equation}\label{alpha money}
\alpha |_{X_{T}     }  = (-1)^{d_1 d_2} \cdot   \frac{ \mu_{X_T} }{  \pi_2^*\mu_{X_{T_2}} } 
\end{equation}
as local sections of $\Omega^\mathrm{top}_{ X_T / X_{T_2} }$ near the chosen point of  the closed subscheme $X_{T_1}(z^\perp) \subset X_T$.
  If we can show this, then restricting both sides to $X_{T_1}(z^\perp)$ and comparing with \eqref{mid nu_z} will complete the proof of Lemma \ref{lem:first nu}.

Define a morphism 
\[
 \Psi :  X_1^{\mathrm{ns}} \times X_2 \to  Y_1^{\mathrm{ns}}   \times M
\]
by $\Psi( y, z) = ( \tau_1(y) , h_{12}(y,z) )$, related to  the $\Phi$ above and the moment morphisms $\tau : X \to Y$ and $\tau_2: X_2 \to Y_2$ by the  diagram
\[
\begin{tikzcd}
{ X_T } \ar[r]  \ar[d, " {\pi_2} " ' ] & {  X_1^{\mathrm{ns}} \times X_{T_2}}  \ar[d , " \Phi" ' ]  \ar[rr , hookrightarrow  ] & &   {  X_1^{\mathrm{ns}} \times X_2 }  \ar[d , " {  (y,z )  \mapsto (\Psi(y,z), z)  } " ' ]   \ar[r , hookrightarrow , " \eqref{Xsplit}"]    & {  X } \ar[dd, " \tau" ]   \\
{ X_{T_2}  } \ar[d] \ar[r, " {  z \mapsto (T_1,0,z)} " ]  &   { ( Y_1^{\mathrm{ns}}\times M  ) \times X_{T_2}  } \ar[rr , hookrightarrow ]   \ar[d]   &  &  {  ( Y_1^{\mathrm{ns}}  \times M)   \times X_2  }  \ar[d , "  {  ( \mathrm{id} , \tau_2 ) }  " ' ]   \\
{  \Spec(F)  } \ar[r , " { (T_1,0 )  } " ' ] &   {      Y_1^{\mathrm{ns}}  \times M    }  \ar[rr ,  " {  (S_1,A) \mapsto (S_1,A, T_2)  } "  ' ]  &   &   {    (Y_1^{\mathrm{ns}}  \times M ) \times Y_2       }  \ar[r , hookrightarrow , " { \eqref{Ysplit} } " ' ]  & Y 
\end{tikzcd}
\]
with cartesian squares.

The equality  \eqref{alpha factorization for Phi} can be rewritten as 
\begin{align*}
\mu_{X_1^{\mathrm{ns}}}\boxtimes\mu_{X_{T_2}}
 &  =
\alpha  \wedge \Phi ^*  \big(\mu_{Y_1^{\mathrm{ns}}}\boxtimes\mu_M  \boxtimes \mu_{X_{T_2}}  \big)    \\
 &  =
\alpha  \wedge \Psi ^*  \big(\mu_{Y_1^{\mathrm{ns}}}\boxtimes\mu_M \big)     \wedge \mathrm{pr}_2^* \mu_{X_{T_2}}  \\
& = 
(-1)^{d_1 d_2}  \cdot  \alpha    \wedge \mathrm{pr}_2^* \mu_{X_{T_2}}  \wedge \Psi ^*  \big(\mu_{Y_1^{\mathrm{ns}}}\boxtimes\mu_M \big)  
\end{align*}
where  $\mathrm{pr}_2 :  X_1^{\mathrm{ns}} \times X_{T_2} \to X_{T_2}$ is the projection to the second factor, and we have used
\[
(-1)^{ \dim(Y_1 \times M) \dim(X_{T_2}) } = (-1)^{d_1 d_2} .
\]
Taking the quotient of both sides by $\Psi ^*  \big(\mu_{Y_1^{\mathrm{ns}}}\boxtimes\mu_M \big)$ and restricting to $X_T$ shows that 
\begin{equation}\label{X_T quotient 1}
 \frac{  \mu_{X_1^{\mathrm{ns}}}\boxtimes\mu_{X_{T_2}} }  { \Psi^*  \big(\mu_{Y_1^{\mathrm{ns}}}\boxtimes\mu_M\big) }  \Big|_{X_T} 
 =  (-1)^{d_1 d_2}  \cdot  \alpha|_{X_T}    \wedge \pi_2^* \mu_{X_{T_2}}  .
\end{equation}

Working locally on neighborhoods in $X$ of points of $X_T$, we may find a form $\beta$ such that 
\begin{equation}\label{beta form}
\mu_X = \beta \wedge \tau^*\mu_Y .
\end{equation}
By definition of $\mu_{X_T}$, we then have
\begin{equation}\label{beta X_T}
\mu_{X_T} =   \frac{\mu_X }{\tau^*\mu_Y}  \Big|_{X_T}  =  \beta|_{X_T} . 
\end{equation}
Using the commutativity of the rightmost square in the diagram above, we see that
\[
( \tau^*\mu_Y ) \big|_{ X_1^\mathrm{ns} \times X_2  }   = 
  \Psi^*( \mu_{Y_1^\mathrm{ns} } \boxtimes \mu_M ) \wedge  \tau_2^*\mu_{Y_2}.
\]
There is an implicit  factor of $(-1)^{\dim (M)\dim (Y_2)}$ in this equality because the  factors in \eqref{Ysplit} have  been reordered in the diagram, but this sign is $1$ because
$\dim (M) =   2 d_1 d_2$. 
 Restricting both sides of \eqref{beta form} along the inclusion $X_1^\mathrm{ns} \times X_2 \to X$, we obtain
\[
\mu_{X_1^\mathrm{ns}}  \boxtimes \mu_{X_2} 
 = \big(  \beta|_{X_1^\mathrm{ns} \times X_2}  \big) \wedge \Psi^*( \mu_{Y_1^\mathrm{ns} } \boxtimes \mu_M ) \wedge  \tau_2^*\mu_{Y_2} .
\]
Taking the quotient of both sides by $\tau_2^* \mu_{Y_2}$ and restricting to $X_1^\mathrm{ns} \times X_{T_2}$, we obtain
\[
\mu_{X_1^\mathrm{ns}}  \boxtimes \mu_{X_{T_2}}  
= \big(  \beta|_{X_1^\mathrm{ns} \times X_{T_2} }  \big)   \wedge \Psi^*( \mu_{Y_1^\mathrm{ns} } \boxtimes \mu_M )  .
\]
Now taking the quotient by $\Psi^*( \mu_{Y_1^\mathrm{ns} } \boxtimes \mu_M )$ and restricting to $X_T$ shows that 
\begin{equation}\label{X_T quotient 2}
 \frac{  \mu_{X_1^{\mathrm{ns}}}\boxtimes\mu_{X_{T_2}} }  { \Psi^*  \big(\mu_{Y_1^{\mathrm{ns}}}\boxtimes\mu_M\big) }  \Big|_{X_T} 
 =  \beta|_{X_T}  \stackrel{ \eqref{beta X_T}}{= }\mu_{X_T} .
\end{equation}

Comparing  \eqref{X_T quotient 1} and  \eqref{X_T quotient 2} proves the equality 
\[
\mu_{X_T} =  (-1)^{d_1 d_2}  \cdot  \alpha|_{X_T}    \wedge \pi_2^* \mu_{X_{T_2}} 
\]
of top-degree forms on $X_T$. Taking the quotient of both sides by $ \pi_2^* \mu_{X_{T_2}}$ yields \eqref{alpha money}, completing the proof.
\end{proof}

\begin{proof} [Proof of Theorem \ref{thm:reduction formula}]
Recall from \eqref{auxiliary fiber form} the top-degree form 
\[
\nu_z   \in H^0\big(  X_{T_1}(z^\perp) ,  \Omega^{ \mathrm{top} }  _{  X_{T_1}(z^\perp)  }  \big) 
\]
associated to any $z \in X_{T_2}(F)$.

 By  Lemma \ref{lem:first nu} we have
\[
  \frac{\mu_{X_T} }{  \pi_2^*\mu_{ X_{T_2} } }  \Big|_{X_{T_1}( z^\perp) }  = (-1)^{d_1 d_2}  \cdot \nu_z ,
\]
and the  sign discrepancy $(-1)^{d_1 d_2}$ disappears when one passes to the associated measures on the $F$-analytic manifold $X_{T_1}(z^\perp)(F)$.  The   quotient form on the left hand side  is the one appearing in the   Fubini-style equality of   \cite[Theorem 7.6.1]{Igusa}  for integration along the fibers of the map $\pi_2$ in  \eqref{factor double fiber}, and therefore
\[
\int_{X_T(F)} f(x) \, d\mu_{X_T}(x) 
= 
\int_{ X_{T_2}(F) }  
\left(  \int_{X_{T_1}( z^\perp) (F) }  f(y) \, d\nu_z (y) \right) \, d\mu_{X_{T_2}}  (z) 
\]
for any locally constant compactly supported function $f: X_T(F) \to \C$. 
In the inner integral we  are identifying each point  $y \in X_{T_1}( z^\perp) (F)$  with its image under the morphism  \eqref{zinc}.  Thus $f(y)$ really means the value of $f$ at  the concatenated tuple $(y,z) \in X_T(F)$.

We apply the above equality of integrals with $f$ equal to the  characteristic function of the compact open subset $C\cap X_T(F)$.  
Directly from Definition \ref{def:density}, we have 
\[
\Den_T(C) 
= 
\frac{\mu_Y( \Herm_d(\co_{F'})^\vee )  }    {  \mu_X(C) }  
\int_{X_T(F)} f(x) \, d\mu_{X_T}(x) .
\]
For  our choice of $f$,  we have $f(y)=1$ if and only if  the concatenated tuple $(y,z) \in V^d$ lies in $C$, which is equivalent to $y \in C_{z^\perp}$.  Thus the inner integral in the iterated integral above is 
\[
 \int_{X_{T_1}( z^\perp) (F) }  f(y) \, d\nu_z (y) 
 = 
  \nu_z \big( C_{z^\perp} \cap X_{T_1}(z^\perp)(F) \big) .
\]
Putting this all together, we find
\[
\Den_T(C) 
= 
\frac{\mu_Y( \Herm_d(\co_{F'})^\vee )  }    {  \mu_X(C) }  
\int_{ X_{T_2}(F) }  
   \nu_z \big( C_{z^\perp} \cap X_{T_1}(z^\perp)(F) \big)    \, d\mu_{X_{T_2}}  (z)   .  
\]

On the other hand, by Lemma \ref{lem:second nu}  we have the equality 
\[
  \frac{ \mu_{X_1 (z^\perp)^\mathrm{ns} } }{ \tau_1^* \mu_{Y_1^\mathrm{ns}}  }  \Big|_{X_{T_1}( z^\perp) }   = \nu_z 
\] 
of top-degree forms on $X_{T_1}( z^\perp)$. 
This quotient form  is exactly the one whose associated measure was used in  Definition \ref{def:density} to define the representation density $\Den_{T_1}(C_{z^\perp})$, and rearranging that definition gives the equality
\[
 \nu_z \big (    C_{z^\perp} \cap  X_{T_1}( z^\perp) ( F)  \big) 
 =
\frac{  \Den_{T_1}(C_{z^\perp}) \cdot   \mu_{X_1(z^\perp)}(C_{z^\perp}) } { \mu_{Y_1}(  \Herm_{d_1}(\co_{F'})^\vee  ) }
\]
(if $C_{z^\perp} = \emptyset$ then both sides are $0$). 
Substituting this expression into the integral above proves Theorem \ref{thm:reduction formula}.
\end{proof}

%%%%%%%%%%%%%%%%%%%%%

\subsection{A special case of the reduction formula}

%%%%%%%%%%%%%%%%%%%%

We state the special case of Theorem \ref{thm:reduction formula} that will be used in the calculations of \S \ref{s:density formulas}.  As in the previous subsection,   fix 
a block diagonal Hermitian matrix 
\[
T  = \begin{pmatrix}  T_1 & \\ & T_2  \end{pmatrix} \in \Herm_d (F')
\]
with each $T_i \in \Herm_{d_i}( F')$  nonsingular.
Assume  further that $T_2$ satisfies $\det(T_2) \in \co_{F}^\times$,  that  $F'/F$ is unramified, and that the residue characteristic of $F$ is not $2$.

The following proposition generalizes \cite[Proposition 9.3]{KR1}, which proves the same statement in the special case $N_1=N_2$.
%The analogous formula for  quadratic spaces is  \cite[Corollary 5.6.3]{Kitaoka}.

\begin{proposition}\label{prop:reduction}
Suppose  we are given $\co_{F'}$-lattices $N_2 \subset N_1$ in $V$ on which the Hermitian form is $\co_{F'}$-valued.  If there exists a tuple  
\[
z=(z_1,\ldots, z_{d_2}) \in N_2^{d_2}
\]
 with moment matrix $T_2$, then 
the isometry class of the $\co_{F'}$-Hermitian lattice
\[
N_1(z^\perp)  = \{ y \in N_1 : h(y,z_i)=0,  \forall  \, 1 \le i \le d_2 \}
\]
depends only on $N_1$ and $T_2$ (not on the particular tuple $z$), and  
\[
\Den_T (N_1^{d_1} \times N_2^{d_2}) =  \Den_{T_1}( N_1(z^\perp) ^{d_1} ) \cdot \Den_{T_2}(N_2^{d_2}).
\]
If no such $z$ exists then  $\Den_T (N_1^{d_1} \times N_2^{d_2}) =0 $.
\end{proposition}

\begin{proof}
We  apply Theorem \ref{thm:reduction formula} with 
\[
C  = N_1^{d_1} \times N_2^{d_2}  \subset V^d.
\]
To understand the integral on the right-hand side of the theorem, 
suppose we have a  point $z \in X_{T_2}(F)$.  In other words, suppose   $z \in V^{d_2}$ is a tuple with moment matrix $T_2$.   In the notation of Theorem \ref{thm:reduction formula}, the subset $C_{z^\perp} \subset V^{d_1}$ is 
\begin{equation*}
C_{z^\perp} =  
\begin{cases}
N_1(z^\perp)^{d_1}  & \mbox{if } z \in  N_2^{d_2}, \\
\emptyset & \mbox{ otherwise.}
\end{cases}
\end{equation*}
Here $N_i (z^\perp) = N_i \cap V(z^\perp)$ is the $\co_{F'}$-submodule of vectors in the lattice $N_i \subset V$ that are orthogonal to all components of the tuple $z$.

If $z \in  N_2^{d_2}$, then  the moment matrix $\tau(z)=T_2$ has integral entries and unit determinant, and so 
the submodule $N_2(z) \subset N_2$ generated by $z_1,\ldots, z_{d_2}$ is a self-dual Hermitian lattice.
It follows that $N_2(z)$ is an orthogonal direct summand of any Hermitian lattice into which it embeds, giving us  orthogonal decompositions
\begin{align}\label{N_1 splitting}
N_1 & =  N_1(z^\perp)  \oplus  N_2(z)   \\
N_2  & = N_2(z^\perp)   \oplus N_2(z)   .  \nonumber 
\end{align}
 The isometry class of $N_2(z)$   depends on the matrix $T_2$, not on the particular tuple $z$.
 The same independence from $z$ therefore holds for  $N_1(z^\perp)$ and $N_2(z^\perp)$, because Witt's cancellation theorem holds for Hermitian $\co_{F'}$-lattices.
This integral version of Witt's cancellation theorem is a consequence (this is the reason we assume $\mathrm{char}(F) \neq 2$) of Jacobowitz's classification of Hermitian lattices over local fields \cite{Jacobowitz}.
See  \cite[\S 2]{Yu} for a short summary of Jacobowitz's results.

 If there is no tuple in $N_2^{d_2}$ with moment matrix $T_2$, then there is no tuple in $N_1^{d_1} \times N_2^{d_2} \subset V^d$   with moment matrix $T$.  
Hence  $ \Den_T( N_1^{d_1} \times N_2^{d_2} )  =0$  by Remark \ref{rem:no represent}, and we are done.

Henceforth we assume there is at least one $z \in N_2^{d_2}$ with moment matrix $T_2$. 
By what we have said, the representation density $ \Den_{T_1}(  N_1(z^\perp)^{d_1}  )$ is independent of $z$.
Factoring this out of the integral in  Theorem \ref{thm:reduction formula}  leaves us with 
\begin{align}\label{density splitting integral}
 \Den_T( C)    & = 
 \frac{  \Den_{T_1}(  N_1(z^\perp)^{d_1}  ) }{ \mu_X( C) } \cdot 
 \frac{ \mu_Y( \Herm_d(\co_{F'}) )  } {  \mu_{Y_1}(  \Herm_{d_1}(\co_{F'})  )  }   \\
 & \quad \times
 \int_{ X_{T_2}(F) } 
   \bm{1}_{ N_2^{d_2} }(z)   \cdot  \mu_{X_1 (z^\perp) } ( N_1(z^\perp)^{d_1}  )  
 \,  d \mu_{X_{T_2}}  (z) . \nonumber
\end{align}
Here $\bm{1}_{ N_2^{d_2} }$  is the characteristic function of $N_2^{d_2} \subset V^{d_2} = X_2(F)$.
Note that we have used the equality
\[
\Herm_r(\co_{F'})    = \Herm_r(\co_{F'})  ^\vee ,
\]
which holds by our assumption that $F'/F$ is unramified.

Now recall the normalization of Haar measure $\mu_{X_1 (z^\perp) }$ on $ V( z^\perp)^{d_1}$.  
By \eqref{X_1(z) quotient measure}, we have
\[
\mu_{X_1 (z^\perp) } ( N_1(z^\perp)^{d_1}  )  
= \frac{\mu_{X_1  } ( N_1(z^\perp)^{d_1}  \times N_2(z)^{d_1} )  }{\mu_M(h_z(N_2(z)^{d_1} ))} .
\]
By \eqref{N_1 splitting}, the numerator on the right-hand side is 
\[
\mu_{X_1  } ( N_1(z^\perp)^{d_1}  \times N_2(z)^{d_1} ) 
=
\mu_{X_1  } ( N_1^{d_1}  ) 
= 
\frac{ \mu_X( C)  }{  \mu_{X_2}(  N_2^{d_2} )   } .
\]
Because $N_2(z)$ is a self-dual Hermitian lattice spanned by the components of $z$, the isomorphism \eqref{h_z iso} restricts to a bijection 
\[
N_2(z)^{d_1} \iso \mathrm{Mat}_{d_1\times d_2}( \co_{F'} )  \subset M(F) ,
\]
and so  the denominator is equal to
\begin{align*}
\mu_M(h_z(N_2(z)^{d_1} )) 
& = \mu_M(    \mathrm{Mat}_{d_1\times d_2}( \co_{F'} )  )  \\
& = \frac{ \mu_Y( \Herm_d(\co_{F'}) )  } {  \mu_{Y_1}(  \Herm_{d_1}(\co_{F'})  )   \cdot    \mu_{Y_2}(  \Herm_{d_2}(\co_{F'})  )  } .
\end{align*}
Combining all of this and plugging it back into the integral in  \eqref{density splitting integral} leaves
\begin{align*}
 \Den_T( C)    & = 
 \Den_{T_1}(  N_1(z^\perp)^{d_1}  )    \\
 & \quad \times
 \frac{    \mu_{Y_2}(  \Herm_{d_2}(\co_{F'})  )     }{  \mu_{X_2}(  N_2^{d_2} )   }  
 \int_{ X_{T_2}(F) }    \bm{1}_{ N_2^{d_2} }(z)    \,  d \mu_{X_{T_2}}  (z)  \\
 & =  \Den_{T_1}(  N_1(z^\perp)^{d_1}  )   \cdot \Den_{T_2}(N_2^{d_2}) ,
\end{align*}
where the second equality is directly from  Definition \ref{def:density}.
\end{proof}

%%%%%%%%%%%%%%%%%%%%%%%%%%%

\section{Explicit densities for $2\times 2$ matrices}
\label{s:density formulas}

%%%%%%%%%%%%%%%%%%%%%%%%%%%%

Throughout   \S \ref{s:density formulas} we assume that $F'/F$ is an unramified  quadratic extension of finite extensions of $\Q_p$, with $p$ odd.
Let   $(V,h)$ be a finite dimensional nondegenerate Hermitian space over $F'$.

Given $\co_{F'}$-lattices  $M,N \subset V$,  our goal is to find explicit formulas for the representation densities of $2\times 2$ matrices with respect to the  compact open subset   $M \times N \subset V^2$, at least in special cases.   
The main results are Corollaries \ref{cor:main 2x2} and \ref{cor:poly 2x2}.

%%%%%%%%%%%%%%%%%%%%%%%%%%%%%%%

\subsection{Some classical  representation densities}

%%%%%%%%%%%%%%%%%%%%%%%%%%%%%%%

Denote by  $k'/k$  the extension of residue fields of $F'/F$, so that $k$ and $k'$ have cardinality $q$ and $q^2$, respectively.
We prove some formulas for  classical representation densities that will be needed in the calculations of the following subsections.  There is nothing here that cannot be deduced from results of Nagaoka  \cite{Nagaoka} or Hironaka  \cite{Hir}.

Given a (possibly degenerate) finite dimensional Hermitian space $A$  over $k'$,   and  a Hermitian matrix $T\in \Herm_d(k')$, set
\[
\mathrm{Rep}_T(A) = \{ a \in A^d :  \tau(a) = T \} .
\]
Here $\tau(a) \in \Herm_d(k')$ is the moment matrix as in \eqref{moment matrix}.

\begin{lemma}\label{lem:finite hermitian reps}
In the notation above, if   $T\in \Herm_d(k')$ is nonsingular and $d \le \dim_{k'}(A)$, then 
\[
\# \mathrm{Rep}_T(A)=  q^{2d\dim_{k'}(A) -d^2}  \prod_{i=1}^d  \left( 1 - (-q^{-1})^{\dim_{k'}( A / A_\mathrm{rad}  )  -i+1} \right)
\]
where $A_\mathrm{rad} \subset A$ is  the radical of the Hermitian form on $A$.
\end{lemma}

\begin{proof}
See the proof of \cite[Theorem 3.5.1]{LZ1}.
%This is elementary linear algebra. 
%By noting that there is  an orthogonal decomposition $A=A_\mathrm{red}\oplus (A/A_\mathrm{red})$,  one can easily reduce to the case where $A$ is nondegenerate.
%Up to isomorphism, there is a unique nondegenerate Hermitian space over $k'$ of any  dimension $D$.  Call it $A_D$, and abbreviate $\rho_T(D) = \# \mathrm{Rep}_T(A_D)$. 
%
%For any  $t\in \Herm_1(k')=k$, one has the trivial calculation
%\[
%\rho_t(1) = \begin{cases}
%1 &\mbox{if }t=0 \\
%q+1 & \mbox{if }t\neq 0.
%\end{cases}
%\]
%Using the recursion relation 
%\[
%\rho_t(D+1)= \sum_{a+b=t} \rho_a(1)\rho_b(D) ,
%\]
%we find that
%\begin{align*}
%\rho_t(D)  = 
%q^{2D-1} \left( 1 - (-q^{-1})^D \right)  + 
%\begin{cases}
%0   & \mbox{if } t\neq 0 \\
% (-q)^D    & \mbox{if } t = 0 .
%\end{cases}
%\end{align*}
%This proves the desired formula  when $d=1$
%
%To compute $\rho_T(D)$ for a general nonsingular $T\in \Herm_d(k')$, one  reduces to the case in which  $T=\mathrm{diag}(t_1,\ldots, t_d)$ is a diagonal matrix, and proves 
%\[
%\rho_T(D)   = \prod_{i=1}^d \rho_{t_i}(D-i+1 ) .
%\]
%Using this, the  general case follows from the $d=1$ case.
\end{proof}

\begin{proposition}\label{prop:nonsingular T density}
Suppose $M\subset V$ is an $\co_{F'}$-lattice on which the Hermitian form is $\co_{F'}$-valued.
 For any  $ T \in \Herm_d(\co_{F'})$ with $\det(T) \in \co_F^\times$ we have
\[
\Den_T(M^d) =      \prod_{i=\epsilon}^{d+\epsilon-1}   ( 1 - (-q^{-1})^{ \dim_{F'}(V) -i  } )  ,
\]
where $\epsilon$ is the $k'$-dimension of the radical of the Hermitian form on the reduction $M_k=M / \pi M$.
\end{proposition}

\begin{proof}
We apply  Proposition \ref{prop:integral densities}, taking  our integral models of the $F$-schemes $X=V^d$ and $Y=\Herm_d(F')$ to be the obvious ones, whose functors of points  are
\[
\mathcal{X}(R)= ( M \otimes_{\co_F} R)^d \quad \mbox{and} \quad \mathcal{Y}(R) = \Herm_d(\co_{F'} \otimes_{\co_F} R)
\]
for any $\co_F$-algebra $R$.
If we denote by   $\mathcal{U} \subset \mathcal{X}$ the  Zariski open subset of  tuples with invertible moment matrix,   the same proof as in  Lemma \ref{lem:tangent moment} proves the smoothness of  the moment morphism $\tau : \mathcal{X} \to \mathcal{Y}$ when restricted to  $\mathcal{U}$.

Taking $C=\mathcal{X}(\co_F)$ in  Proposition \ref{prop:integral densities}, we find that
\[
\Den_T(M^d) 
 =  \#  \mathcal{U}_T( k )  
\cdot
\frac{ \# \mathcal{Y}( k )   }{  \# \mathcal{X}(k)   }  .
\]
The cardinality
$
 \#  \mathcal{U}_T( k ) =  \# \mathrm{Rep}_T(M_k)
$
is known by Lemma \ref{lem:finite hermitian reps}, and of course
$\#\mathcal{Y}(k)=q^{d^2}$ and $\#\mathcal{X}(k)=q^{2d \dim_{F'}(V) }$.
The claim follows by combining  these formulas.
\end{proof}

\begin{proposition}\label{prop:vector density}
Suppose  $M \subset V$ is an $\co_{F'}$-lattice that is self-dual under the Hermitian form.
For any $t\in  \co_F=\Herm_1(\co_{F'})$ we have 
\[
\frac{ \Den_t(M)  }{  \Den_1(M)   } =  \sum_{i=0}^{\ord(t)} (-1)^i   ( - q^{-1})^{  i ( \dim_{F'}(V)  -1 ) }  ,
\]
where 
\[
 \Den_1(M)   =    1  -  (-q^{-1})^{  \dim_{F'}(V)  } . 
\]
\end{proposition}

\begin{proof}
The first formula is   \cite[Example 3.5.2]{LZ1} in the case of a $1\times 1$ Hermitian matrix.
The second is the $1\times 1$ case of Proposition \ref{prop:nonsingular T density}.
\end{proof}

%%%%%%%%%%%%%%%%%%%%%%%%%%%%%%%

\subsection{A preliminary density formula}

%%%%%%%%%%%%%%%%%%%%%%%%%%%%%%%

Fix $\co_{F'}$-lattices $M,N \subset V$ on which the Hermitian form is $\co_{F'}$-valued, and  a nonsingular Hermitian matrix
\[
T = \begin{pmatrix} t_1 & * \\ *  & t_2  \end{pmatrix}  \in \Herm_2(\co_{F'}) .
\]
We would like to understand the representation density of $T$ with respect to the compact open subset $M\times N \subset V^2$.
This seems quite hard in general, but Proposition \ref{prop:main 2x2} below provides  a reasonably explicit way to compute $\Den_T(M\times N)$  under the following assumption.

\begin{hypothesis}
Throughout this subsection  we assume:
\begin{enumerate}
\item
The lattice $M$ is self-dual under the Hermitian form.
\item
$t_1,t_2 \in \co_F^\times.$
\end{enumerate}
\end{hypothesis}

Recall from  \eqref{XandY} that $X=V^2$ , but viewed as an $F$-scheme.
The choices of $M$ and $N$ (now viewed as $\co_F$-lattices)  determine a smooth integral model of $X$, namely  $\mathcal{X} = M\times N$,  viewed as an $\co_F$-scheme.  More precisely, $\mathcal{X}$ is defined by its functor of points
\[
\mathcal{X}(R) = M_R \times N_R
\]
 for any $\co_F$-algebra $R$.

Recall from  \eqref{XandY}  that $Y=\Herm_2(F')$,  viewed as an $F$-scheme.  The $\co_F$-lattice $\Herm_2(\co_{F'})$ determines a smooth integral model of $Y$, but this is not the right thing to look at in this context, as the moment matrix of a pair $(x_1,x_2) \in M \times N$ need not  have integral entries. 

 Instead, fix a uniformizer $\pi \in \co_F$, let  $i =i (M,N)$ be the  integer characterized by 
\begin{equation}\label{idef}
\pi^i N \subset M  \quad \mbox{and}\quad \pi^{ i -1}  N \not\subset M ,
\end{equation}
and   let $\mathcal{Y}$ be the integral model of $Y$ determined by the $\co_F$-lattice 
\begin{equation}\label{denominator Y}
\left\{
 \begin{pmatrix}
t_1 &    \pi^{-i} s \\ 
\pi^{-i} s^\sigma & t_2
\end{pmatrix}  : t_1,t_2\in \co_F,\, s\in \co_{F'}
\right\}  \subset \Herm_2(F').
\end{equation}
More precisely,  define an $\co_F$-scheme $\mathcal{Y}$ with functor of points 
\[
\mathcal{Y} (R) = R \times R \times R'
\]
for any $\co_F$-algebra $R$.  When $R$ is an $F$-algebra we use 
\[
(t_1,t_2,s) \mapsto   \begin{pmatrix}
t_1 &    \pi^{-i} s \\ 
\pi^{-i} s^\sigma  & t_2
\end{pmatrix} 
\]
to identify $\mathcal{Y}(R) \iso Y(R)$.
This fixes an isomorphism of the generic fiber of $\mathcal{Y}$ with $Y$, in such a way that $\mathcal{Y}(\co_F) \subset Y(F)$ is identified with \eqref{denominator Y}.

Now recall the moment map $\tau : X \to Y$, which on $F$-points sends  $x=(x_1,x_2) \in V^2$ to its moment matrix $\tau(x) \in \Herm_2(F')$.
The moment map  admits a unique extension to 
\begin{equation}\label{technical integral moment}
\tau : \mathcal{X} \to \mathcal{Y}.
\end{equation}
Indeed, if we denote by 
\[
( x_2\mapsto \widetilde{x}_2  ) : N_R  \to  M_R 
\]
the $R'$-linear map induced by $\pi^i : N   \to  M $, the extension sends
$
(x_1,x_2) \in   \mathcal{X}(R)
$
 to the triple $(T_1,T_2,B) \in \mathcal{Y}(R)$ with $T_j = h(x_j,x_j)$ and 
$
B = h( x_1 , \widetilde{x}_2)  .
$

\begin{lemma}\label{lem:smoothness}
 Let $\kappa$ be an algebraically closed field with a ring map $\co_F \to \kappa$, and set $\kappa' = \kappa \otimes_{\co_F} \co_{F'}$.  If 
\[
x=(x_1,x_2) \in  M_\kappa \times N_\kappa=  \mathcal{X}(\kappa)
\]
is a geometric point of $\mathcal{X}$ such that
\begin{enumerate}
\item
the $\kappa'$-submodule of $M_\kappa$ generated by $x_1$ and $\widetilde{x}_2$ is free of rank $2$, and
\item
there is some $\nu \in N_\kappa$ such that $h(\nu , x_2) \in (\kappa')^\times$, 
\end{enumerate}
then the  moment morphism  induces a surjection  on tangent spaces
\[
T_{ \mathcal{X}, x} \to T_{ \mathcal{Y} , \tau(x) }.
\]
\end{lemma}

\begin{proof}
If the field $\kappa$ is an extension of  $F$,  this follows from Lemma \ref{lem:tangent moment}.
Thus  we may assume $\kappa$ is an extension of the residue field $k=\co_F/\pi \co_F$.

The map on tangent spaces
\[
M_\kappa \oplus N_\kappa  = T_{ \mathcal{X}, x} \to T_{ \mathcal{Y} , \tau(x) }  =   \kappa \oplus \kappa \oplus  \kappa' ,
\]
is equal to the composition
\[
M_\kappa \oplus N_\kappa  \to  \kappa' \oplus \kappa' \oplus  \kappa'  \to  \kappa \oplus \kappa \oplus  \kappa' 
\]
in which the first arrow is the $\kappa$-linear map defined by
\begin{align*}
( \mu ,  0  )   &  \mapsto  \big(    h( \mu ,x_1)  ,  0   , h( \mu ,  \widetilde{x}_2 )     \big)  \\
(  0  , \nu )    &  \mapsto  \big(  0   ,  h(\nu ,x_2)    ,     h(   x_1, \nu   )     \big)  ,
\end{align*}
and the second is $(a,b,c) \mapsto ( \mathrm{Tr}(a) , \mathrm{Tr}(b) , c)$.
Here $\mathrm{Tr} : \kappa' \to \kappa$ is the trace.  Thus it suffices to show that the first arrow is surjective.

Using the nondegeneracy of the Hermitian form on $M_\kappa$,  our assumption on the submodule generated by $x_1$ and $\widetilde{x}_2$  guarantees that the linear functionals $h(\cdot, x_1)$ and $h(\cdot,\widetilde{x}_2)$ on $M_\kappa$ are $\kappa'$-linearly independent. 
 It follows that the image of $M_\kappa \oplus 0$ under the first arrow in the composition is $\kappa'\oplus 0 \oplus \kappa'$.
Our assumption that $h(\nu , x_2) \in (\kappa')^\times$ for some choice of $\nu$  guarantees that the image of $0 \oplus N_\kappa$ contains an element of the form $(0,1, * ) \in \kappa'\oplus \kappa' \oplus \kappa'$.  The  surjectivity of the first arrow in the composition follows immediately.
\end{proof}

\begin{lemma}\label{lem:pre 2x2}
Assume that  $N \not\subset M$, so that  the integer $i$ of \eqref{idef} is positive.
 The compact open subset
\[
C = \{ (x_1,x_2) \in M\times N  :  \pi^{i-1} x_2 \not\in M \} \subset V^2
\]
satisfies
\begin{align*}
\frac{   \Den_T(C)     \Vol(C)  } {   \Vol(  M \times N  )   }  
&  =   \frac{     1 - (-q^{-1})^{D-2}  } {     q^{   2   D +  2  i - 1    }    } 
   \left( \#\mathrm{Rep}_{t_2}( N_k) - \# \mathrm{Rep}_{t_2}( A ) \right) .
\end{align*}
Here $D=\dim_{F'}(V)$,  and $A$ is the kernel of the map $N_k \to M_k$
obtained by reducing $\pi^i : N \to M$.  
\end{lemma}

\begin{proof}
For any $\co_F$-algebra $R$, define
$
\mathcal{U}(R) \subset M_R \times N_R
$
to be the set of pairs $(x_1,x_2)$ satisfying 
\begin{itemize}
\item
$x_1 , \widetilde{x}_2 \in M_R$ generate an $R'$-module  local direct summand of rank $2$,
\item
there is some $\nu \in N_R$ such that  $h(\nu , x_2) \in (R')^\times$.
\end{itemize}
  This functor is represented by an open subscheme $\mathcal{U}  \subset  \mathcal{X}$, 
and Lemma \ref{lem:smoothness}  implies that the  restriction to $\mathcal{U}$ of the moment morphism \eqref{technical integral moment} is smooth.
As in \eqref{integral cartesian},
define $\co_F$-schemes $\mathcal{U}_T  \subset \mathcal{X}_T$  as  fiber products
\[
\begin{tikzcd}
{ \mathcal{U}_T } \ar[r] \ar[d]   & { \mathcal{X}_T }  \ar[r] \ar[d]  & { \Spec(\co_F) } \ar[d , "T" ]  \\
{ \mathcal{U} }   \ar[r]  & { \mathcal{X} } \ar[r , "\tau" '] & { \mathcal{Y} } .
\end{tikzcd}
\]

   Let us write 
\[
T =   \begin{pmatrix} t_1 & \pi^{-i}(\pi^i s)  \\  \pi^{-i} (\pi^i s^\sigma )& t_2  \end{pmatrix} 
=(t_1,t_2, \pi^i s) \in \mathcal{Y}(\co_F) 
\]
for some $s \in \co_{F'}$.  Because we assume $i>0$, the reduction of $T$ to $\mathcal{Y}(k)$ is $(t_1,t_2,0)$, where we denote again by $t_i$ its reduction to $k^\times$.  
Unpacking the definition of the moment morphism \eqref{technical integral moment},
the pairs $(x_1,x_2) \in M_k \times N_k = \mathcal{X} (k)$ that lie in $\mathcal{X}_T(k)$ are precisely those that satisfy
\[
h(x_1,x_1)=t_1 , \quad h(x_2,x_2)=t_2 , \quad    h(x_1,\widetilde{x}_2)=0.
\]
Moreover, any such pair must satisfy
\begin{equation}\label{x_2 isotropic}
h ( \widetilde{x}_2,\widetilde{x}_2) = \pi^{2i} h (x_2,x_2) =0.
\end{equation}
From this it is not hard to see that 
\begin{equation}\label{basic U_T}
  \mathcal{U}_T(k)  = \left\{
(x_1,x_2) \in M_k \times N_k  : 
\begin{array}{c}  
h(x_1,x_1)=t_1 , \, h(x_2,x_2)=t_2    \\
\widetilde{x}_2 \neq 0 , \, h(x_1,\widetilde{x}_2)=0 
 \end{array} 
\right\} . 
\end{equation}

A point of $\mathcal{X}_T(\co_F)$ is just a pair $x=(x_1,x_2) \in M\times N$ with moment matrix $\tau(x)=T$.  
Given that $h(x_2,x_2) \in \co_{F'}^\times$,   such a pair lies in $\mathcal{U}_T(\co_F)$ if and only if $x_1, \pi^i x_2 \in M$ generate an $\co_{F'}$-module local direct summand of rank $2$.   This is equivalent to the condition that 
the reductions of $x_1$ and $\widetilde{x}_2$ to $M_k$ are $k'$-linearly independent. 
In other words, $\mathcal{U}_T(\co_F)$ is equal to the preimage of $\mathcal{U}_T(k)$ under the reduction map 
$\mathcal{X}_T(\co_F) \to \mathcal{X}_T(k)$.  Given the description \eqref{basic U_T}, we find that
\[
 \mathcal{U}_T(\co_F)  = \left\{
(x_1,x_2) \in M  \times N  :  \tau(x) =T \mbox{ and }  \pi^i x_2 \not\in \pi M \right\} .
\]
In other words, $ \mathcal{U}_T(\co_F)  = C \cap X_T(F)$.

We have now verified the hypotheses of Proposition \ref{prop:integral densities}, which implies
\[
\Den_T(C) 
 =  \#  \mathcal{U}_T( k )  
\cdot
\frac{ \# \mathcal{Y}( k )   }{  \# \mathcal{X}(k)   }
\cdot \frac{   \Vol( \mathcal{X}(\co_F) )   }{    \Vol(C)  }
\cdot \frac{ \Vol(  \Herm_2(\co_{F'})  ) }{ \Vol(  \mathcal{Y}(\co_F)  ) } .
\]
Note that we have used $ \Herm_2(\co_{F'}) = \Herm_2(\co_{F'})^\vee$, a consequence of the assumption that $F'/F$ is unramified imposed throughout \S \ref{s:density formulas} .

Clearly $\# \mathcal{Y}(k) = q^4$ and  $\# \mathcal{X}(k) = q^{ 4 D}$.  
As $\Herm_2(\co_{F'}) \subset \mathcal{Y}(\co_F)$ is an $\co_{F}$-submodule of colength $2 i$, we have 
\[
\frac{ \Vol(  \Herm_2(\co_{F'})  ) }{ \Vol(  \mathcal{Y}(\co_F)  ) } = q^{-2 i}.
\]
Plugging this information into the above formula for $\Den_T(C)$ shows that 
 \[
\frac{   \Den_T(C)     \Vol(C)  } {   \Vol(  M \times N  )   }  
 =   \frac{   \# \mathcal{U}_T(k)  } {     q^{   4   D   + 2 i - 4    }    } .
\]

It remains to count the number of elements in \eqref{basic U_T}.    
The number of $x_2 \in N_k$ such that $\widetilde{x}_2 \neq 0$ and $h( x_2,x_2) = t_2$  is equal to the number of times $N_k$ represents $t_2$, 
minus the number of times that $A=\ker( N_k \to M_k)$ represents $t_2$. 
For each such $x_2$, the nonzero vector $\widetilde{x}_2 \in M_k$ is isotropic, by \eqref{x_2 isotropic},  so the number of $x_1\in M_k$ satisfying $h(x_1,\widetilde{x}_2)=0$ and $h(x_1,x_1) = t_1$ is  the number of times $t_1$ is represented by the degenerate Hermitian space  $\langle \widetilde{x}_2 \rangle^\perp$ of vectors in $M_k$  orthogonal to $\widetilde{x}_2$.
This degenerate Hermitian space has  $k'$-dimension $D-1$, and radical of   dimension $1$.
Thus,  for each of the $\#\mathrm{Rep}_{t_2}( N_k) - \# \mathrm{Rep}_{t_2}( A )$ choices of $x_2$, there are 
\[
 \# \mathrm{Rep}_{t_1}( \langle \widetilde{x}_2 \rangle^\perp ) = q^{2D-3  }   \left( 1 - (-q^{-1})^{D-2} \right)
\]
choices of $x_1$.  This proves that 
\[
 \# \mathcal{U}_T(k)   = q^{2D-3  }   \left( 1 - (-q^{-1})^{D-2} \right) \cdot \left( \#\mathrm{Rep}_{t_2}( N_k) - \# \mathrm{Rep}_{t_2}( A ) \right) ,
\]
completing the proof.
\end{proof}

The following result provides a formula for  $\Den_T(M\times N)$ in terms of the representation numbers of finite Hermitian spaces from  Lemma \ref{lem:finite hermitian reps}.
 For it to be useful, one must know   the relative positions of the lattices $M$ and $N$.
We will explore this in the next subsection.

\begin{proposition}\label{prop:main 2x2}
Define  $\co_{F'}$-lattices 
\[
N^{(i)} = \{ x_2 \in N : \pi^i x_2 \in M\}
\]
for every $i\ge 0$, and let  $m\ge 0$ be the integer determined by 
\[
N \cap M = N^{(0)} \subsetneq N^{(1)}   \subsetneq  \cdots  \subsetneq N^{(m)} =N .
\]
 The representation density  $\Den_T( M\times N)$ is equal to 
\begin{align*}
 &   \frac{     \left( 1 -  (-q^{-1}) ^{ D- \epsilon   }    \right)    \left(   1  -  (-q^{-1}) ^{ D-1  }\right)     }{ [ N : N^{(0)}  ]   }     
\sum_{\ell =0}^{\ord_F( \det(T) )} (-1)^\ell   (-q^{-1})^{  \ell ( D -2 ) }   \\
 & \quad    +
  \left(  1 - (-q^{-1})^{D-2}  \right)    \sum_{i=1}^m    \frac{  \#\mathrm{Rep}_{t_2}( N^{(i)}_k ) - \# \mathrm{Rep}_{t_2}( A^{(i)})   } {    q^{   2   D +  2  i - 1    }  [ N : N^{(i)}  ]    } .
\end{align*}
Here $D=\dim_{F'}(V)$, $\epsilon$ is the $k'$-dimension of the radical of $N^{(0)}_k$, and 
\[
A^{(i)} = \mathrm{ker}( N_k^{(i)} \to M_k) 
\]
 is  the kernel of the map  obtained by reducing $\pi^i : N^{(i)} \to M$. 
\end{proposition}

\begin{proof}
 Decompose
$
M\times N = C_0 \sqcup C_1 \sqcup \cdots \sqcup C_m
$
by setting
\[
C_i = \begin{cases}
M \times N^{(0)}  & \mbox{if } i=0 \\
M \times ( N^{(i)}  \smallsetminus N^{(i-1) })  & \mbox{if } i >0 .
\end{cases}
\]
Remark \ref{rem:general disjoint} implies
\[
\Den_T( M \times N) 
%=   \sum_{i=0}^m \frac{ \Den_T(C_i)  \Vol(C_i)}{\Vol(M\times N)}  
=   \sum_{i=0}^m \frac{ \Den_T(C_i)  \Vol(C_i)}{ [ N : N^{(i)}  ]  \Vol(M\times N^{(i)})}  ,
\]
and we will compute the terms on the right-hand side individually.
For $i >0 $, applying Lemma \ref{lem:pre 2x2} with $N$ replaced by $N^{(i)}$ shows that
\begin{align*}
  \frac{ \Den_T(C_i)   \Vol(C_i)}{      \Vol(M\times N^{(i)} )}   
  =   \frac{     1 - (-q^{-1})^{D-2}  } {    q^{   2   D +  2  i - 1    }     } 
\left(  \#\mathrm{Rep}_{t_2}( N^{(i)}_k ) - \# \mathrm{Rep}_{t_2}( A^{(i)})   \right) ,
\end{align*}
and so we are done if we can prove
\begin{align}\label{reduction application}
\Den_T (C_0) & =   \left( 1 -  (-q^{-1}) ^{ D- \epsilon   }   \right) \left(   1  -  (-q^{-1}) ^{ D-1  }\right)   \\
& \quad \times \sum_{\ell =0}^{\ord( \det(T)  )} (-1)^\ell   (-q^{-1})^{  \ell ( D -2 ) } . \nonumber
\end{align}

Endow  $\co_{F'}^2$ with the Hermitian form determined by our fixed 
\[
T = \begin{pmatrix} t_1 & * \\ *  & t_2  \end{pmatrix}  \in \Herm_2(\co_{F'}) ,
\]
and denote by  $f_1 ,f_2 \in  \co_{F'}^2$  the standard basis vectors.
Because $t_2 \in \co_{F'}^\times$, we can apply the Gram-Schmidt process to find an orthogonal basis  of the form $e_1  , f_2 \in \co_{F'}^2$; that is to say, without changing $f_2$.  
In particular, the  change of basis matrix $g\in \GL_2(\co_{F'})$ is lower triangular, and the matrix 
\[
T \action g  = \begin{pmatrix}  s_1 &  \\  & t_2 \end{pmatrix} \in \Herm_2(\co_{F'} ) 
\]
 from \eqref{change of basis}  is  diagonal with $\ord_F(s_1) = \ord_F( \det(T))$.
 Because  $N^{(0)} \subset M$ and $g$ is lower triangular, we have  $C_0g = C_0$.   Hence
\[
\Den_T(C_0) = \Den_{ T\action g}( C_0  ) 
\]
by Proposition \ref{prop:density invariance}.

Applying the reduction formula  from Proposition \ref{prop:reduction} to the diagonal matrix $T\action g$, we find that 
\[
\Den_T (C_0) =  \Den_{s_1}( M(z^\perp)  ) \cdot \Den_{t_2}( N^{(0)} ), 
\]
where $M(z^\perp)  \subset M$ is the set of vectors orthogonal to some $z\in N^{(0)}$ of Hermitian norm $t_2$. 
 If $N^{(0)}$ does not represent $t_2$,  then  $\Den_T (C_0)=0$.

Proposition \ref{prop:nonsingular T density} implies 
\[
\Den_{t_2}( N^{(0)} ) =        1 -  (-q^{-1}) ^{ D- \epsilon   } ,
\]
and in particular the right-hand side is $0$ if  $N^{(0)}$ does not represent $t_2$.
Assuming the existence of some $z$ as above,   $M(z^\perp)$ is a self-dual Hermitian $\co_{F'}$-lattice of rank $D-1$, and so  Proposition \ref{prop:vector density} implies
\[
\Den_{s_1}( M(z^\perp)  )  = \left(   1  -  (-q^{-1}) ^{ D-1  }\right)  
\sum_{\ell =0}^{\ord(s_1)} (-1)^\ell   (-q^{-1})^{  \ell ( D -2 ) }  .
\]
The conclusion is that, regardless of whether or not $N^{(0)}$  represents $t_2$,  the equality \eqref{reduction application} holds.
\end{proof}

%%%%%%%%%%%%%%%%%%%%%%%%%%%%%

\subsection{The density polynomial in a simple case}

%%%%%%%%%%%%%%%%%%%%%%%%%%%%%

We continue to work with a nonsingular Hermitian matrix
\[
T = \begin{pmatrix} t_1 & * \\ *  & t_2  \end{pmatrix}  \in \Herm_2(\co_{F'}) 
\]
satisfying $t_1,t_2 \in \co_{F'}^\times$, and remind the reader that $F'/F$ is assumed to be unramified throughout \S \ref{s:density formulas}. 
Suppose we are given self-dual $\co_{F'}$-lattices $M,N \subset V$, with the property that
\begin{equation}\label{close lattices}
M /(M \cap N) \iso \co_{F'}/ \pi^m \co_{F'} \iso N/(M \cap N )
\end{equation}
as $\co_{F'}$-modules for some integer $m \ge 0$.

The following two  corollaries of Proposition  \ref{prop:main 2x2}  will be needed in \S \ref{ss:ASW easy case}.
When  $m=0$, so that $M=N$, these formulas agree with known formulas for classical representation densities found in  \cite{Hir,  KR1, LZ1,Nagaoka}.

\begin{corollary}\label{cor:main 2x2}
Suppose  $V$ is a direct sum of $r+1$ hyperbolic planes.
\begin{enumerate}
\item
If $m=0$,  then
\begin{align*}
 \Den_T(M \times N) 
  &  =      \left( 1 -  q^{ - 2r - 2    }    \right)    \left(   1  +   q^{  - 2r - 1  }\right)        
\sum_{\ell =0}^{\ord_F( \det(T) )} (-1)^\ell   q^{  - 2  \ell r } .
\end{align*}
\item
 If $m>0$, then
\begin{align*}
 \Den_T(M \times N) 
  &  = 
q^{-2m}   (  1 - q^{- 2r }  )      (   1  +     q^{ -2r  -1  }  )  
    \sum_{\ell=0}^{\ord_F( \det(T))} (-1)^\ell   q^{  - 2 \ell r  }  \\
& \quad   +
q^{ - 2m}    (  1 - q^{-2r}  )    ( 1-q^{-2}  )     \Big(   m      +  ( 1-m)  q^{ - 2r  }        \Big)   .
\end{align*}
\end{enumerate}
\end{corollary}

\begin{proof}
We just need to compute the various quantities appearing in the  formula of Proposition  \ref{prop:main 2x2}, 
which are determined by the relation \eqref{close lattices}.

If   $0<i\le m$,  then easy calculations show that
$
[ N : N^{(i)} ] = q^{2(m-i)} ,
$
and that the $k'$-dimension of the radical of $N^{(i)}_k$  is $2$ if $i<m$, and $0$ if $i=m$.
The $k'$-Hermitian space $A^{(i)}$ has dimension $2r+1$, and its radical has dimension $1$.
Plugging this information into Lemma \ref{lem:finite hermitian reps},  we find
\begin{align*}
 \#\mathrm{Rep}_{t_2}( N^{(i)}_k ) - \# \mathrm{Rep}_{t_2}( A^{(i)})
  =   
   q^{   4r+ 3        }   ( 1 - q^{-2} ) 
\begin{cases}
 (  1 - q ^{ -2r }  )      & \mbox{if }  i   < m  \\
1    & \mbox{if }  i  = m  .
\end{cases}
\end{align*}
The desired formula for the density now follows from Proposition  \ref{prop:main 2x2}.
\end{proof}

We now restate Corollary \ref{cor:main 2x2} in terms of the density polynomials defined in  Proposition \ref{prop:density polynomial}.

\begin{corollary}\label{cor:poly 2x2}
Suppose  $V$ is a  hyperbolic plane.
\begin{enumerate}
\item
If $m=0$,  then
\begin{align*}
 \Den_T(s, M \times N) 
  &  =      \left( 1 -  q^{ - 2s - 2    }    \right)    \left(   1  +   q^{  - 2s - 1  }\right)        
\sum_{\ell =0}^{\ord_F( \det(T) )} (-1)^\ell   q^{  - 2  \ell s } .
\end{align*}
\item
 If $m>0$, then
\begin{align*}
 \Den_T(s, M \times N) 
  &  = 
q^{-2m}   (  1 - q^{- 2s }  )      (   1  +     q^{ -2s  -1  }  )  
    \sum_{\ell=0}^{\ord_F( \det(T))} (-1)^\ell   q^{  - 2 \ell s }  \\
& \quad   +
q^{ - 2m}    (  1 - q^{-2s}  )    ( 1-q^{-2}  )     \Big(   m      +  ( 1-m)  q^{ - 2s  }        \Big)   .
\end{align*}
\end{enumerate}
\end{corollary}

\begin{proof}
Using the notation of \S \ref{ss:density polynomial},  Corollary \ref{cor:main 2x2}  applies to  the self-dual lattices $M^{[r]} = M \oplus \Lambda_{r,r}$ and $N^{[r]} = N \oplus \Lambda_{r,r}$ in the direct sum of   $r+1$ hyperbolic planes   $V^{[r] } = V \oplus W_{r,r}$.  This gives an explicit formula for     
\[
\Den_T(r, M\times N ) = \Den_T( M^{[r]} \times N^{[r]} ),
\]
proving  that the desired equalities of polynomials in $q^{-s}$ hold whenever $s$ is a positive integer.
They therefore  hold for all $s$.
\end{proof}

%%%%%%%%%%%%%%%%%%%%%%%%%%%

\section{Arithmetic Siegel-Weil}
\label{s:ASW}

%%%%%%%%%%%%%%%%%%%%%%%%%%%%

In this section we formulate our conjectural extension of the arithmetic Siegel-Weil formula, incorporating the action of certain correspondences constructed by Li-Rapoport-Zhang.  We then verify a special case of this conjecture.
 
Throughout \S \ref{s:ASW} we work in  the situation of \S \ref{ss:introASW},  so that $F'/F$ is an unramified  quadratic extension  of finite extensions of $\Q_p$, with $p$ odd.
  Fix an integer $n\ge 2$.
 Up to isomorphism, there are two  nondegenerate $F'$-Hermitian spaces of dimension $n$, which we call $V$ and $\mathbb{V}$.
They are distinguished by 
\[
\eta(\det(V)) = 1 \quad \mbox{and}  \quad  \eta(\det(\mathbb{V})) = -1, 
\]
where  $\eta : F^\times   \to \{ \pm 1\}$ is the unique unramified quadratic character.

%%%%%%%%%%%%%%%%%%%%%%%%%%%

\subsection{Kudla-Rapoport cycles}

%%%%%%%%%%%%%%%%%%%%%%%%%%%%

We quickly recall the Kudla-Rapoport cycles  on unitary Rapoport-Zink spaces, following \cite{KR1} and \cite{LZ1}.

Denote by  $k'/k$  the extension of residue fields of $F'/F$, and fix a uniformizer $\pi \in \co_F$.
Denote by  $W$  the completion of the maximal unramified extension of $\co_{F'}$.  Its residue field $W/\pi W$  is an algebraic closure of $k'$.

\begin{definition}
A \emph{signature $(1,n-1)$ Hermitian $\co_{F'}$-module}  over a $\mathrm{Spf}(W)$-scheme $S$ is a triple $(X,\iota,\lambda)$ in which 
\begin{itemize}
\item $X$ is a formal $\pi$-divisible $\co_F$-module over $S$ of relative height $2n$ and dimension $n$,
\item $\iota :\co_{F'} \to \End(X)$ is an action extending the $\co_F$-action, and satisfying the Kottwitz  condition of signature $(1,n-1)$: for all $a \in \co_{F'}$ the endomorphism $\iota(a)$ of the locally free $\co_S$-module $\Lie(X)$ has characteristic polynomial $(T-a)(T-\sigma(a))^{n-1} \in \co_S[T]$.
\item
$\lambda : X \iso X^\vee$ is a principal polarization whose associated Rosati involution on $\End(X)$ restricts to $\sigma \in \Gal(F'/F)$ on the image of $\iota$.
\end{itemize}
For ease of notation, we usually just write $X$ in place of the triple $(X,\iota,\lambda)$.
\end{definition}

Up to quasi-isogeny (compatible with the $\co_{F'}$-action and polarization) there is a unique signature $(1,n-1)$ Hermitian $\co_{F'}$-module.  
Let us fix one such $\mathbb{X}$  over $W/\pi W$.  Denote by $\mathcal{N}_n$ the associated formally smooth $n$-dimensional Rapoport-Zink formal scheme over $\mathrm{Spf}(W)$,  parametrizing pairs $(X, \rho)$  consisting of a signature $(1,n-1)$ Hermitian $\co_{F'}$-module $X$ defined over a $\mathrm{Spf}(W)$-scheme $S$, together with an $\co_{F'}$-linear  quasi-isogeny 
\[
\rho : X\times_S \overline{S} \dashrightarrow \mathbb{X} \times_{ W/\pi W   } \overline{S}
\]
of height $0$ that respects the polarizations.
Here $\overline{S}$ is the reduction of $S$ to a scheme over $W/\pi W$.

Now fix a signature $(1,0)$ Hermitian $\co_{F'}$-module $X_0$ over $W$, and denote by $X^\sigma_0$ the same polarized formal $\co_F$-module, but with  the $\co_{F'}$-action $i_{X_0}$ replaced by $i_{X_0} \circ \sigma$.
Thus  $X^\sigma_0$ satisfies the Kottwitz condition of signature $(0,1)$.
Denote by $\mathbb{X}_0$ and $\mathbb{X}_0^\sigma$ the reductions of $X_0$ and $X_0^\sigma$ to $W/\pi W$.

As  in \cite[\S 3]{KR1} and \cite[\S 2.2]{LZ1}, the $\co_{F'}$-module $\Hom_{\co_{F'}}(  \mathbb{X} ^\sigma _0 , \mathbb{X} )$ carries a natural Hermitian form, and 
\[
\mathbb{V} =  \Hom_{\co_{F'}}(   \mathbb{X} ^\sigma _0 , \mathbb{X} ) \otimes  \Q 
\]
is  the unique $n$-dimensional Hermitian space with $\eta(\det(\mathbb{V}))=-1$.

As  in \cite[\S 2.3]{LZ1},   any  nonzero vector $x \in \mathbb{V}$ determines a Cartier divisor $\mathcal{Z}(x) \subset \mathcal{N}_n$, whose $S$-points are those pairs $(X, \rho) \in \mathcal{N}_n(S)$ for which the composition 
\[
\rho^{-1}   \circ x   \in \Hom_{\co_{F'} } (  \mathbb{X}^\sigma _0 \times_{   W/\pi W   } \overline{S}   , X\times_S \overline{S}    )  \otimes \Q
\]
is the reduction to $\overline{S}$ of a  morphism $ X^\sigma _0 \times_W S  \to  X$  of formal $\co_{F'}$-modules over $S$.

More generally, any $d$-tuple $x \in \mathbb{V}^d$ with linearly independent components determines a closed formal subscheme
\[
\mathcal{Z}(x)   \define   \mathcal{Z}(x_1) \cap \cdots \cap \mathcal{Z}(x_d)  \subset \mathcal{N}_n.
\]
These \emph{Kudla-Rapoport cycles} are typically not equidimensional,  so one instead works with the derived intersection 
\begin{equation}\label{derived KR}
 \mathcal{Z}{}^\mathbb{L}(x) 
   =  [  \co_{ \mathcal{Z}(x_1) } \otimes^{\mathbb{L}} \cdots  \otimes^{\mathbb{L}} \co_{ \mathcal{Z}(x_d) } ]
  \in    K_0^{ \mathcal{Z}(x)} ( \mathcal{N}_n )
\end{equation}
as in \cite[Definition 2.4.1]{LZ1}.  
Here $K_0^{ \mathcal{Z}(x)} ( \mathcal{N}_n )$ is the Grothendieck group of bounded complexes of vector bundles on $\mathcal{N}_n$,  whose cohomology sheaves are formally supported on $\mathcal{Z}(x)$; see    \cite[\S B.1]{Zhang2} for the formalism of these $K$-groups.

Both $\mathcal{Z}(x)$ and $\mathcal{Z}{}^\mathbb{L}(x)$ depend only on the $\co_{F'}$-lattice in $\mathbb{V}$ spanned by the  components of $x$, not on $x$ itself.  This is proved in \cite{How}, although because of our standing assumption that $F'/F$ is  unramified,  there is a more elementary proof \cite[Corollary 2.8.2]{LZ1}, generalizing the special case proved in  \cite[Proposition 3.2]{Ter}.

Now consider the case $d=n$.
If $x \in \mathbb{V}^n$ has nonsingular moment matrix (i.e. the components of $x$ form a basis of $\mathbb{V}$), then   any  complex $ \mathcal{F}^\bullet \in K_0^{ \mathcal{Z}(x)} ( \mathcal{N}_n )$
has a well-defined Euler-Poincar\'e characteristic
\[
\chi(    \mathcal{F}^\bullet)  =   \sum_{i,j} (-1)^{i+j} \mathrm{length}_W \,  H^j(  \mathcal{N}_n  , \mathcal{H}^i( \mathcal{F}^\bullet )  ),
\]
where $\mathcal{H}^i(\mathcal{F}^\bullet)$ is the $i^\mathrm{th}$ cohomology sheaf of the complex.
The essential point is that all cohomology groups appearing here are of finite length as $W$-modules,  because $\mathcal{H}^i(\mathcal{F}^\bullet)$ is formally supported on $\mathcal{Z}(x)$, which is  a proper scheme over $W/\pi^\ell W$ for some  $\ell \gg 0$.

%%%%%%%%%%%%%%%%%%%%%%%%%%%

\subsection{The spherical Hecke algebra}
\label{ss:spherical}

%%%%%%%%%%%%%%%%%%%%%%%%%%%%

Now we switch to the  Hermitian space $V$, characterized by $\eta(\det(V))=1$, and recall some constructions from \cite{LRZ}.

Denote by  $\mathrm{U}(V)$  the  group of Hermitian isometries of $V$, regarded as a locally compact totally disconnected group (as opposed to an algebraic group over $F$).
Fix a basis of $V$ for which the Hermitian form is given by the  antidiagonal unit matrix, and use this to regard $\mathrm{U}(V) \subset \GL_n(F')$.
  The $\co_{F'}$-span of the chosen basis vectors  is a self-dual lattice 
  \[
  \Xi \subset V,
  \]
  and we denote by $K \subset \mathrm{U}(V)$ the maximal compact subgroup of isometries that stabilize $\Xi$.

Denote by  
\[
\mathcal{H}(V) = C_c ( K \backslash \mathrm{U}(V) / K , \Q  )
\]
the spherical Hecke algebra of  compactly supported $K$-bi-invariant $\Q$-valued functions on $\mathrm{U}(V)$.
For any even integer $2 \le t \le n$, define the \emph{standard Hecke function}  $f_t^\mathrm{std}\in \mathcal{H}(V)$ as  the characteristic function of the double coset 
\[
K \begin{pmatrix}
\pi I_{t/2} & & \\
& I_{n-t} & \\
& & \pi^{-1} I_{t/2}
\end{pmatrix} K  \subset \mathrm{U}(V).
\]
The $\Q$-algebra  $\mathcal{H}(V)$ is isomorphic to a polynomial algebra in $\lfloor n/2 \rfloor$ variables, and the standard Hecke functions form a polynomial basis.  For notational convenience, we extend the definition also to $t=0$ by letting $f_0^\mathrm{std}\in \mathcal{H}(V)$ be the characteristic function of $K$.

In \cite[Definition 4.1.1]{LRZ}, one finds the construction of  \emph{atomic Hecke functions}  
\[
f_t^\LRZ  \in \mathcal{H}(V),
\]
(denoted $\varphi_t$ in loc.~cit.) indexed again by even integers $2 \le t \le n$.   
Each atomic Hecke function has  the form
\[
f^\LRZ_t = f_t^\mathrm{std} + \sum_{   \substack{ 0 \le t' < t \\  t' \    \mathrm{even} }}  m(t',t) f_{t'}^\mathrm{std}
\]
for integers $m(t',t)$, and so the atomic Hecke functions also form a polynomial basis of $\mathcal{H}(V)$.

The point of introducing these atomic Hecke functions is that, unlike the standard Hecke functions,  their construction has a geometric analogue:  for every even integer $2 \le t \le n$, there is an  \emph{atomic Hecke correspondence}
\begin{equation}\label{atomic correspondence}
\mathbb{T}_t^\LRZ : K_0( \mathcal{N}_n ) \to  K_0( \mathcal{N}_n)
\end{equation}
on the Grothendieck group of complexes of vector bundles on the Rapoport-Zink formal scheme $\mathcal{N}_n$ defined in the previous subsection.
These are the correspondences denoted by $\mathbb{T}_t$  in  \cite[ (5.5.3)]{LRZ}.
It is conjectured by Li-Rapoport-Zhang that the atomic Hecke correspondences pairwise commute.
Assuming this conjecture, one obtains  a  $\Q$-algebra homomorphism 
\[
\mathcal{H}(V)  \map{f_t^\LRZ \mapsto  \mathbb{T}^\LRZ_t}  \End( K_0( \mathcal{N}_n)_\Q  )  .
\]

To avoid assuming this commutativity conjecture, we introduce a noncommutative polynomial ring $\widetilde{\mathcal{H}}(V)$ in variables $y_t$,  indexed by even integers $2\le t\le n$, and define  $\Q$-algebra maps
\begin{equation}\label{atomic map}
\begin{tikzcd}
&  { \widetilde{\mathcal{H}}(V) }   \ar[dl  ,  "   y_t \mapsto f_t^\LRZ  "  ' ]  \ar[dr , " y_t\mapsto  \mathbb{T}^\LRZ_t" ]     \\
{ \mathcal{H}(V)  }  & & { \End( K_0( \mathcal{N}_n)_\Q  )  . } 
\end{tikzcd}
\end{equation}
Of course when  $n=2$ we have $\widetilde{\mathcal{H}}(V)  = \mathcal{H}(V)$.

Let $A \subset \mathcal{N}_n$ be any closed formal subscheme.  Using  a geometric correspondence 
\[
\begin{tikzcd}
&   { \mathcal{T} ^{\le t } }  \ar[dl] \ar[dr] \\ 
{  \mathcal{N}_n }  & & {  \mathcal{N}_n , }
\end{tikzcd}
\]
 Li-Rapoport-Zhang define another closed formal subscheme $\mathcal{T}^{ \le t} (A) \subset \mathcal{N}_n$, and refine  \eqref{atomic correspondence}  to a homomorphism
\[
\mathbb{T}_t^\LRZ : K^A_0( \mathcal{N}_n ) \to  K^{ \mathcal{T} ^{ \le t} (A)  } _0( \mathcal{N}_n).
\]
Directly from the definition \cite[(5.5.1)]{LRZ}, one sees that for any $d$-tuple $x \in \mathbb{V}^d$ one has
$
\mathcal{T} ^{ \le t} ( \mathcal{Z}(x) )   \subset \mathcal{Z}(\pi^2 x) .
$
In this way, any $f \in \widetilde{\mathcal{H}}(V)$ determines a homomorphism
\begin{equation}\label{good correspondence support}
f : K^{ \mathcal{Z}(x) }_0( \mathcal{N}_n ) \to  K^{ \mathcal{Z} (\pi^r x)  } _0( \mathcal{N}_n)
\end{equation}
for all $r\gg 0$.

%%%%%%%%%%%%%%%%%%%%%%%%%%%

\subsection{The conjecture}

%%%%%%%%%%%%%%%%%%%%%%%%%%%%

For any nonsingular $T\in \Herm_n(F')$ and any $\varphi\in S(V^n)$ we have the Whittaker function $W_T(s, \varphi)$ of \S \ref{ss:whittaker}, determined by an unramified additive character $\psi : F \to \C^\times$. 
 If $\varphi$ is the characteristic function of a compact open subset $C \subset V^n$, then 
\[
W_T(s , \mathbf{1}_C ) =  \Vol(C) \cdot \Den_T(s,C)
\]
by Proposition \ref{prop:whittaker}.  Note that  the Haar measure on $V$ in that proposition is the one for which our fixed self-dual lattice $\Xi \subset V$ has volume $1$.

If  $x \in \mathbb{V}^n$ is a tuple whose moment matrix  $T=\tau(x)$ is nonsingular, then   $\eta(\det(T)) = -1$.
This implies that there is no tuple  in $V^n$  with moment matrix $T$, 
and so by Remark \ref{rem:no represent}, we have $\Den_T(C)=0$ for \emph{every} compact open subset $C \subset V^n$.  
Equivalently,   $W_T(s, \varphi)$ vanishes at $s=0$ for every choice of $\varphi\in S(V^n)$.
In the special case where $\varphi$ is the characteristic function of $\Xi^n$, the arithmetic Siegel-Weil formula, conjectured by Kudla-Rapoport \cite{KR1} and proved by Li-Zhang \cite{LZ1}, relates the derivative of the Whittaker function at $s=0$ to the Euler-Poincar\'e characteristic
of the derived  Kudla-Rapoport cycle \eqref{derived KR}.

\begin{theorem}[Arithmetic Siegel-Weil  \cite{LZ1}]\label{thm:ASW}
If $x \in \mathbb{V}^n$ is any tuple whose moment matrix $T=\tau(x)$ is nonsingular, then 
\[
\chi ( \mathcal{Z}^\mathbb{L} (x)  )   \cdot   c(n) \cdot   \log(q^2)  
=W'_T(0, \bm{1}_{ \Xi^ n }  ) ,
\]
where $c(n)=  \prod_{i=1}^n (   1- (-q)^{-i}  )$.
\end{theorem}

For any $d \ge 1$, the Hecke algebra $\mathcal{H}(V)$  acts on the space  $S(V^d)^K$ of $K$-fixed Schwartz functions via   
\[
(f * \varphi)(x) =  \frac{1}{ \Vol(K)}  \int_{ \mathrm{U}(V) }  f (g)  \varphi(g^{-1} x) \, dg .
\]
This induces an action also of $\widetilde{\mathcal{H}}(V)$, via the leftward arrow in \eqref{atomic map}.
It is natural to expect that a stronger version of Theorem \ref{thm:ASW} also holds, incorporating this action and the action of $\widetilde{\mathcal{H}} (V)$ on $K_0(\mathcal{N}_n)$ via the rightward arrow in \eqref{atomic map}.
For example, for any $x \in \mathbb{V}^n$ with nonsingular moment matrix $T$, and any $f \in \widetilde{\mathcal{H}}(V)$,  one expects
\begin{equation}\label{pre conjecture one}
\chi  (  f*  \mathcal{Z} {}^\mathbb{L} (x)  )  
\stackrel{?}{=}     \frac{  W'_T \big(0, f* \bm{1}_\Xi^{\otimes n} \big) }{   c(n)\cdot \log(q^2)}      .
\end{equation}
Because the Euler-Poincar\'e characteristic is not well defined on $K_0( \mathcal{N}_n)_\Q$, the $f*$ on the left should be understood as the endomorphism   \eqref{good correspondence support}.

In the same spirit, given a tuple  $x=(x_1,\ldots, x_n) \in \mathbb{V}^n$ with  nonsingular moment matrix $T$, and  any $f_1,\ldots, f_n \in \widetilde{\mathcal{H}}(V)$,   one expects
\begin{equation}\label{pre conjecture two}
\chi \big( f_1  \mathcal{Z} {}^\mathbb{L}(x_1)  \otimes^{\mathbb{L}}  \cdots \otimes^{\mathbb{L}}  f_n \mathcal{Z} {}^\mathbb{L}(x_n)  \big)
\stackrel{?}{=} 
\frac{  W'_T  \big(0, f_1 \bm{1}_\Xi \otimes \cdots \otimes  f_n \bm{1}_\Xi \big) } { c(n)\cdot   \log(q^2)   }  
\end{equation}
where we now omit the $*$'s for simplicity.  Here the left-hand side is understood as above: using \eqref{good correspondence support}, we obtain classes 
\[
f_i \mathcal{Z}^\mathbb{L}(x_i) \in  K^{ \mathcal{Z} (\pi^r x_i)  } _0( \mathcal{N}_n)
\]
for some $r \gg 0$.  The tensor product of these classes lives in $K^{ \mathcal{Z} (\pi^r x)  } _0( \mathcal{N}_n)$, so has a well-defined Euler-Poincar\'e characteristic.

\begin{remark}
Note that \eqref{pre conjecture one} and \eqref{pre conjecture two} are not equivalent.  
The  action of $\mathcal{H}(V)$ on Schwartz functions is not compatible with tensor products, in the sense that under the canonical isomorphism
\[
S(V^{d_1})^K \otimes S(V^{d_2})^K \iso S( V^{d_1+d_2} )^K
\]
one typically has
\[
( f  \varphi_1) \otimes (  f \varphi_2) \neq  f ( \varphi_1 \otimes \varphi_2)   .
\]
There is a similar lack of compatibility of the Hecke correspondences under tensor products of classes in  $K_0$-groups.
\end{remark}

Interpolating between the  two extreme cases \eqref{pre conjecture one} and \eqref{pre conjecture two}, we make the following conjecture.

\begin{conjecture}\label{conj:sphericalASW}
Let $d_1+\cdots+d_r =n$ be a partition of $n$.
Let  $x=(x_1,\ldots, x_r) \in \mathbb{V}^n$ be a concatenation of tuples $x_i \in \mathbb{V}^{d_i}$, and assume that the moment matrix $T=\tau(x)$ is nonsingular.  For any $f_1,\ldots, f_r \in \widetilde{\mathcal{H}}(V)$ we have the intersection formula
\begin{align*}  
 \chi \big( f_1  \mathcal{Z} {}^\mathbb{L}(x_1)  \otimes^{\mathbb{L}}  \ldots \otimes^{\mathbb{L}}   f_r \mathcal{Z} {}^\mathbb{L}(x_r)  \big)
=
 \frac{  W'_T \big(0, f_1 \bm{1}_{\Xi^{d_1}} \otimes \cdots \otimes  f_r \bm{1}_{\Xi^{ d_r} } \big)  }{ c(n)\cdot    \log(q^2)  } .
\end{align*}
\end{conjecture}

Of course Theorem \ref{thm:ASW} is the special case in which all $f_i=1$.

%%%%%%%%%%%%%%%%%%%%%%%%%%%

\subsection{A nontrivial case of the conjecture}
\label{ss:ASW easy case}

%%%%%%%%%%%%%%%%%%%%%%%%%%%%

Now we take $n=2$.  In this case the Hermitian space $V$ is a hyperbolic  plane, and the Rapoport-Zink formal scheme $\mathcal{N}_2$ is the formal spectrum of a power series ring $W[[ t]]$ in one variable.    

Suppose we are given linearly independent vectors  $x_1,x_2 \in \mathbb{V}$ satisfying $h(x_i,x_i) \in \co_F^\times$, and denote by 
\[
T = \begin{pmatrix}
h(x_1,x_1) & h(x_1,x_2) \\
h(x_2,x_1) & h(x_2,x_2) 
\end{pmatrix}
\in \Herm_2( F' ) 
\]
the  corresponding moment matrix.  As 
$
\eta(\det(T)) = \eta(\det(\mathbb{V})) =-1 ,
$
we see that $\ord_F(\det(T))$ must be  odd.  This  implies that $T \in \Herm_2(\co_{F'})$, and that  
\[
\Delta \define  \frac{1+\ord_F(\det(T)) }{2}
\]
 is a positive integer.

Our goal in \S \ref{ss:ASW easy case} is to prove the following  special case of Conjecture \ref{conj:sphericalASW}.
For a nonzero vector $x \in \mathbb{V}$ the Kudla-Rapoport cycle $\mathcal{Z}(x)$ is a Cartier divisor on $\mathcal{N}_2$, and
\[
\mathcal{Z}^\mathbb{L}(x)  =[  \co_{ \mathcal{Z}(x)  }  ]  \in K_0^{ \mathcal{Z}(x) } ( \mathcal{N}_2) 
\]
 is the class in the Grothendieck group of (a resolution by vector bundles of) the structure sheaf of $\mathcal{Z}(x)$, regarded as a coherent sheaf on $ \mathcal{N}_2$.
 We remind the reader that in our case of $n=2$, the leftward arrow in \eqref{atomic map} is an isomorphism $\widetilde{\mathcal{H}}(V) \iso \mathcal{H}(V)$.  Hence we drop the tilde from the notation.

\begin{theorem} \label{thm:SASW}
For any Hecke functions $f_1,f_2 \in \mathcal{H}(V)$ we have
\[
\chi  \big( f_1 \mathcal{Z}^\mathbb{L} (x_1)  \otimes^{\mathbb{L}}  f_2 \mathcal{Z}^\mathbb{L}(x_2)   \big)  \cdot  \log( q^2 )   
= 
  \frac{  W'_T(0, f_1\bm{1}_\Xi  \otimes   f_2\bm{1}_\Xi )   }{  (   1   +    q^{-1}  )  ( 1 - q^{-2} )   } .
  \]
\end{theorem}

We will use a result of Li-Rapoport-Zhang to reduce the calculation of the left-hand side of  Theorem \ref{thm:SASW} to the following result of Kudla-Rapoport.

\begin{proposition}[Kudla-Rapoport] \label{KRdegree}
 For any integers $a\ge b\ge 0$ we have
 \begin{align*}
\chi \big(    \mathcal{Z}^\mathbb{L} (   \pi^a x_1)  \otimes^\mathbb{L}     \mathcal{Z}^\mathbb{L}( \pi^b x_2) \big)
 & = 
 \sum_{\ell=0}^{ 2b} q^\ell (   a+  b    - \ell  +   \Delta   ) .
 \end{align*}
\end{proposition}

\begin{proof}
The assumptions  $h(x_i,x_i) \in \co_F^\times$ and $a \ge b$  imply that  the $\co_{F'}$-lattice spanned by $\pi^a x_1$ and $\pi^b x_2$ admits a basis  with moment matrix 
 \[
  \begin{pmatrix} \pi^{2 a } \det(T) &  0  \\  0   &  \pi^{2b}   \end{pmatrix} .
 \]
It therefore follows from \cite[Theorem 1.1(iv)]{KR1}   that the scheme-theoretic intersection $\mathcal{Z}  ( \pi^a x_1 ) \cap \mathcal{Z}( \pi^b x_2 )$ is Artinian, and 
\begin{align*}
\chi \big(    \mathcal{Z}^\mathbb{L}  (   \pi^a x_1)  \otimes^\mathbb{L}     \mathcal{Z}^\mathbb{L} ( \pi^b x_2) \big)
&  = \deg \big(  \mathcal{Z}  ( \pi^a x_1 ) \cap \mathcal{Z}( \pi^b x_2 )   \big)   \\
& =   \frac{1}{2} \sum_{\ell=0}^{ 2b} q^\ell (  1+ 2a+ 2 b    -2 \ell    + \ord_F(\det(T))    ) ,
 \end{align*}
as desired.
\end{proof}

For any $m \ge 0$, let $f_m^\circ \in \mathcal{H}(V)$ be the characteristic function of the double coset
\begin{equation}\label{cartan coset}
K   \begin{pmatrix} \pi^m &  \\  & \pi^{-m}  \end{pmatrix}   K  \subset \mathrm{U}(V).
\end{equation}
By the Cartan decomposition, the functions $f_0^\circ, f_1^\circ, f_2^\circ, \ldots$ form a $\Q$-basis of $\mathcal{H}(V)$.
By the proof of \cite[Lemma 7.4.1]{LRZ}, they are related to the atomic Hecke function $f_2^\LRZ$  from \S \ref{ss:spherical} by
\[
 f^\LRZ_2  = f^\circ_1 + (q+1) ,
\]
and satisfy the  recursion relation  
\[
 f^\circ_{m+1}  = 
f_2^\LRZ  f^\circ_m  - 2q f^\circ_m -    q^2 f^\circ_{m-1} .
\]

\begin{proposition}\label{prop:intersection 2x2 final}
 For any integers $a\ge b\ge 0$ we have
\[
\chi \big(   f_a^\circ   \mathcal{Z}^\mathbb{L}  (   x_1)  \otimes^\mathbb{L}   f_b^\circ  \mathcal{Z}^\mathbb{L}  ( x_2) \big)   
= \begin{cases}
\Delta  & \mbox{if } a=b=0 \\
 1  & \mbox{if } a>b=0 \\
q^{2b-1}    (   q  + 1  )    \Delta   +   q^{2b-1}          & \mbox{if } a=b >0 \\
q^{2b-1}     (q+1 )   & \mbox{if } a>b>0.
\end{cases}
\]
\end{proposition}

\begin{proof}
If $a=b=0$ then
\[
f_a^\circ   \mathcal{Z}^\mathbb{L}  (   x_1)  \otimes^\mathbb{L}   f_b^\circ  \mathcal{Z}^\mathbb{L}  ( x_2) 
=  \mathcal{Z}^\mathbb{L}  (   x_1)  \otimes^\mathbb{L}     \mathcal{Z}^\mathbb{L}  ( x_2) ,
\]
and the claim is immediate from Proposition \ref{KRdegree}.
Thus we  assume  $a>0$.

Combining  \cite[Proposition 7.4.2]{LRZ} with \cite[(7.4.6)]{LRZ},  if $\mathcal{Z} \subset \mathcal{N}_2$ is any closed formal subscheme whose scheme-theoretic intersection with $\mathcal{Z}(\pi^a x_1)$ is Artinian, then for any class $\mathcal{C} \in K_0^{\mathcal{Z}}( \mathcal{N}_2)$ we have 
\[
\chi \big(    f_a^\circ \mathcal{Z}^\mathbb{L}( x_1) \otimes^\mathbb{L}  \mathcal{C}   \big) 
= 
\chi \big(     \mathcal{Z}^\mathbb{L}( \pi^a x_1) \otimes^\mathbb{L}  \mathcal{C}  \big) 
-\chi \big(    \mathcal{Z}^\mathbb{L}( \pi^{a-1} x_1) \otimes^\mathbb{L} \mathcal{C}   \big)  .
\]
We   apply this with $\mathcal{C}$ equal  to 
\[
f_b^\circ \mathcal{Z}^\mathbb{L} (x_2) \in K_0^{ \mathcal{Z}( \pi^r x_2) } (\mathcal{N}_2)
\]
with $r \gg 0$   to obtain
\begin{align}\label{hecke distribution}
& \chi \big(    f_a^\circ \mathcal{Z}^\mathbb{L}( x_1) \otimes^\mathbb{L}   f_b^\circ \mathcal{Z}^\mathbb{L} (x_2)\big)  \\
&  = 
\chi \big(     \mathcal{Z}^\mathbb{L}( \pi^a x_1) \otimes^\mathbb{L}  f_b^\circ \mathcal{Z}^\mathbb{L} (x_2) \big)   
-  \chi \big(    \mathcal{Z}^\mathbb{L}( \pi^{a-1} x_1) \otimes^\mathbb{L} f_b^\circ \mathcal{Z}^\mathbb{L} (x_2) \big)  .  \nonumber
\end{align}

If $b=0$, so that $f_b^\circ \mathcal{Z}^\mathbb{L}(x_2) = \mathcal{Z}^\mathbb{L}(x_2)$, 
then the claim follows immediately from \eqref{hecke distribution} and Proposition \ref{KRdegree}.
If $b>0$ then the same reasoning behind  \eqref{hecke distribution} also  shows 
\begin{align*}
& \chi \big(     \mathcal{Z}^\mathbb{L}( \pi^a x_1) \otimes^\mathbb{L}  f_b^\circ \mathcal{Z}^\mathbb{L} (x_2) \big)  \\
& =
\chi \big(     \mathcal{Z}^\mathbb{L}( \pi^a x_1) \otimes^\mathbb{L}   \mathcal{Z}^\mathbb{L} ( \pi^b x_2) \big) 
-\chi \big(     \mathcal{Z}^\mathbb{L}( \pi^a x_1) \otimes^\mathbb{L}   \mathcal{Z}^\mathbb{L} ( \pi^{b-1}x_2) \big) ,
\end{align*}
and similarly with $a$ replaced by $a-1$.  Thus  \eqref{hecke distribution} may be   expanded as
\begin{align*}
& \chi \big(  f_a^\circ   \mathcal{Z}^\mathbb{L} (x_1 )   \otimes^\mathbb{L} 
f_b^\circ  \mathcal{Z}^\mathbb{L}  ( x_2)  \big)  \\
 &   = 
 \chi\big(  \mathcal{Z}^\mathbb{L}  ( \pi^a x_1)     \otimes^\mathbb{L}  \mathcal{Z}^\mathbb{L}  ( \pi^b x_2)      \big) 
-  \chi \big(   \mathcal{Z}^\mathbb{L}  ( \pi^a x_1)     \otimes^\mathbb{L}  \mathcal{Z}^\mathbb{L}  ( \pi^{b-1} x_2)  \big)   \\
& \quad  -     \chi \big(  \mathcal{Z}^\mathbb{L} ( \pi^{a-1} x_1)    \otimes^\mathbb{L}   \mathcal{Z}^\mathbb{L} ( \pi^b x_2)  \big) 
+  \chi \big(  \mathcal{Z}^\mathbb{L} ( \pi^{a-1} x_1)    \otimes^\mathbb{L}   \mathcal{Z}^\mathbb{L} ( \pi^{b-1} x_2)  \big) ,
\end{align*}
 and again the claim follows from    Proposition  \ref{KRdegree}.
\end{proof}

It remains to compute $W_T(s, f_1\bm{1}_\Xi  \otimes   f_2\bm{1}_\Xi )$.
We will do this by expressing this Whittaker function as a linear combination of the density polynomials computed in Corollary \ref{cor:poly 2x2}.

Suppose $M,N \subset V$ are self-dual lattices.  Using our assumption that $F'/F$ is unramified, it is easy to see that both $M$ and $N$ admit orthonormal bases, and hence lie in the same $\mathrm{U}(V)$-orbit. 
By the Cartan decomposition of $\mathrm{U}(V)$ into double cosets of the form \eqref{cartan coset}, it follows that there is an $m\in \Z_{\ge 0}$ for which 
\[
M/(M\cap N) \iso \co_{F'} / \pi^m \co_{F'} \iso N/(M\cap N)
\]
as $\co_{F'}$-modules.   
We call $m$ the \emph{relative position} of $M$ and $N$, and write 
\[
M\sim_m N.
\]
 The action of the $\Q$-basis $f_0^\circ,f_1^\circ, \ldots \in \mathcal{H}(V)$  on the characteristic function of our fixed self-dual lattice $\Xi \subset V $  is by 
\begin{equation}\label{hecke xi}
f_m^\circ *  \bm{1}_\Xi = \sum_{ M \sim_m \Xi }  \bm{1}_M.
\end{equation}

\begin{lemma}\label{lem:tree count}
Assume $a\ge b \ge 0$, and for any $m\ge 0$ denote by $r_{a,b}(m)$ the number of ordered pairs $(M,N)$ of self-dual lattices $M,N \subset V$ satisfying the three relations 
\[
M \sim_a \Xi , \quad N\sim_b \Xi , \quad M\sim_m N .
\]
If at least one of $a$, $b$, and $m$ is positive, then
\[
r_{a,b}(m) =
\begin{cases}
 q^{a+b+m }(1+q^{-1})  & \mbox{if } m=a \pm b \\
 q^{  a + b +  m }  (1-q^{-2}) & \mbox{if } a-b <  m < a+b \\
 0 & \mbox{otherwise}.
\end{cases}
\]
Obviously $r_{0,0}(0)=1$.
\end{lemma}

\begin{proof}
This is a combinatorial exercise in counting paths in the Bruhat-Tits tree associated to the group $\mathrm{SU}(V)$.

As in  \cite[\S 4.1]{LRZ}, an $\co_{F'}$-lattice $\Lambda \subset V$ is called a \emph{vertex lattice} if $\Lambda \subset \Lambda^\vee \subset \pi^{-1}\Lambda$, where $\Lambda^\vee$ is the dual lattice relative to the Hermitian form.
The \emph{type} of a vertex lattice is the $\co_{F'}$-length of $\Lambda^\vee/ \Lambda$.  
Because $V$ is a hyperbolic plane, the type must be $0$ or $2$.  
The vertex lattices of type $0$ are the self-dual lattices, and those of type $2$ are lattices satisfying $\Lambda^\vee =\pi^{-1} \Lambda$.

The building of  $\mathrm{SU}(V)$ has the structure of a tree.  Its vertices are the vertex lattices, with two distinct vertex lattices connected by an edge whenever one  is contained within the other.   
Every vertex has $q+1$ edges emanating from it, and every edge connects a type $0$ lattice to a type $2$ lattice.  
In particular, the distance between any two self-dual lattices is even, and their relative position is half the distance between  them.

By a \emph{geodesic} we mean a path in the tree without backtracking.  For any two vertices $\Lambda$ and $\Lambda'$, we denote by $[\Lambda \mapsto  \Lambda']$ the unique geodesic starting at $\Lambda$ and ending at $\Lambda'$.

First consider the case $b=0$.   The integer $r_{a,0}(m)$ counts the number of self-dual lattices $M$ that are at distance both $2a$ and  $2m$ from $\Xi=N$.  This is obviously $0$ if  $m\neq a$, and if $m=a$ it is the number of geodesics  of length $2m$ emanating from $\Xi$.  In other words,
\[
r_{a,0}(m) = 
\begin{cases}
1 & \mbox{if } m=a=0 \\
  q^{2m-1}(q+1)  & \mbox{if } m=a>0 \\
   0 & \mbox{if }m\neq a .
\end{cases}
\]
  From now on we assume $b>0$.

For each pair of self-dual lattices $M$ and $N$ contributing to $r_{a,b}(m)$ there is a unique $0\le d \le  2b $ such that the geodesics $[\Xi \mapsto M]$ and $[\Xi  \mapsto N]$ agree for length $d$.  
The condition that $M \sim_m N$ (i.e.~that the distance from $M$ to $N$ is $2m$)  is equivalent to $d=a+b  -m$.
As  $0 \le d \le 2b$, we deduce
\[
r_{a,b}(m) \neq 0 \implies  a-b \le m \le a+b.
\]

Suppose  $m=a+b$.  
By the previous paragraph, any self-dual lattices $M$ and $N$ contributing to $r_{a,b}(m)$ are the endpoints of geodesics (of lengths $2a$ and $2b$) emanating from $\Xi$ and having no edges in common.  
The number of such pairs of geodesics is
\[
r_{a,b}(m) = q^{2a+2b-1} ( q+1)   = q^{a+b+m-1}(q+1)
\]

Next suppose $m=a-b$.  For any $M$ and $N$ contributing to $r_{a,b}(m)$, the lattice $N$ is determined by $M$: if one forms the geodesic of length $2a$ from $\Xi $ to $M$, then follows it for length $2b$, one arrives at $N$.  Hence
\[
r_{a,b}(m) = q^{2a-1}(q+1)= q^{a+b+m-1}(q+1)
\]
is the number of geodesics of length $2a$ emanating from $\Xi$.

Now suppose $a-b<m< a+b$.
If we set  $d= a+b-m>0$,  so that $0< d < 2b$, we  can uniquely specify a pair $(M,N)$ contributing to $r_{a,b}(m)$ by the following process:
\begin{enumerate}
\item
Choose a geodesic $\gamma_0$ of length $d$ emanating from $\Xi$, and call the endpoint $P$.
There are $q^{d-1}(q+1)$ choices for $\gamma_0$. 

\item
Choose a geodesic $\gamma_1$ of length $2a-d$ emanating from $P$, and having no edge in common with $\gamma_0$.   Call the endpoint $M$.  
There are $q^{2a-d}$ choices for $\gamma_1$. 
\item

Choose a geodesic $\gamma_2$ of length $2b-d$ emanating from $P$, and having no edge in common with $\gamma_0$ or $\gamma_1$.  Call the endpoint $N$.  There are 
$q^{2b-d-1}(q-1)$ choices for $\gamma_2$.
\end{enumerate} 
The total number of ways to choose the paths $\gamma_0$, $\gamma_1$, and $\gamma_2$ is therefore
\[
r_{a,b}(m) =   
q^{  a + b +  m - 2}  (q^2-1) ,
\]
completing the proof.
\end{proof}

\begin{proposition}\label{prop:whittaker 2x2 final}
 For any integers $a\ge b\ge 0$ we have
\[
 W'_T(0, f_a^\circ \bm{1}_\Xi  \otimes   f_b^\circ\bm{1}_\Xi )  
=    c(2)   \log( q^2 )  \begin{cases}
\Delta  & \mbox{if } a=b=0 \\
 1  & \mbox{if } a>b=0 \\
q^{2b-1}    (   q  + 1  )    \Delta   +   q^{2b-1}          & \mbox{if } a=b >0 \\
q^{2b-1}     (q+1 )   & \mbox{if } a>b>0,
\end{cases}
\]
where $c(2)  =  (   1   +    q^{-1}  )  ( 1 - q^{-2} )  $.
\end{proposition}

\begin{proof}
Suppose  $M , N \subset V$ are self-dual lattices in relative position $M \sim_m N$.
Corollary \ref{cor:poly 2x2} gives an explicit formula for
\[
W_T (s ,  \bm{1}_M \otimes \bm{1}_N )  = \Den_T(s, M \times N),
\]
which shows that it only depends on the relative position $m$, not on the particular lattices $M$ and $N$.
Let us therefore abbreviate 
\[
w_T (s,m) =    \Den_T(s, M \times N).
\]
After taking the derivative at $s=0$, and recalling that 
\[
\ord_F \det(T) = 2\Delta - 1 , 
\]
  the formula of Corollary \ref{cor:poly 2x2} simplifies to 
\[
w'_T (0,m )  =  c(2)   \log( q^2 )     
\begin{cases}
  \Delta & \mbox{if } m=0  \\ 
 q^{-2m}    (   1  +  q^{-1}  )^{-1}   & \mbox{if } m>0  .
 \end{cases}
\]

Now use  \eqref{hecke xi} to write
\[
W_T(s, f_a^\circ \bm{1}_\Xi  \otimes  f_b^\circ \bm{1}_\Xi)  
= 
\sum_{ \substack{  M \sim_a \Xi    \\  N \sim_b \Xi   }   } 
W_T(s,   \bm{1}_M \otimes  \bm{1}_N) ,
\]
in which the sum is over all ordered pairs of self-dual lattices $M$ and $N$ in $V$ of the indicated relative positions with $\Xi$.  Collecting  together all  terms for which  $M$ and $N$ have the same relative position shows that 
\[
W'_T(0, f_a^\circ \bm{1}_\Xi \otimes  f_b^\circ \bm{1}_\Xi )  
 = 
\sum_{m \ge 0}
r_{a,b}(m) \cdot  w'_T(0,  m ) ,
\]
where $r_{a,b}(m)$ is the integer of Lemma \ref{lem:tree count}.  
After a bit of elementary manipulation, the claim follows  from our explicit formulas for the quantities on the right-hand side. 
\end{proof}

\begin{proof}[Proof of Theorem \ref{thm:SASW}]
As both sides of the desired equality
\[
\chi  \big( f_1 \mathcal{Z}^\mathbb{L} (x_1)  \otimes^{\mathbb{L}}  f_2 \mathcal{Z}^\mathbb{L}(x_2)   \big)  \cdot  \log( q^2 )   
= 
  \frac{  W'_T(0, f_1\bm{1}_\Xi  \otimes   f_2\bm{1}_\Xi )   }{  (   1   +    q^{-1}  )  ( 1 - q^{-2} )   } 
  \]
  are $\Q$-linear in each of the Hecke functions $f_1 , f_2 \in \mathcal{H}(V)$, we may assume that 
$f_1= f_a^\circ$ and $f_2=f_b^\circ$ for some $a,b\in \Z_{\ge 0}$.  By swapping the factors in the tensor products on both sides if necessary, we may also assume that  $a\ge b$.   The desired equality follows by comparing Proposition \ref{prop:intersection 2x2 final} with Proposition \ref{prop:whittaker 2x2 final}. 
\end{proof}

\bibliographystyle{amsalpha}
%\bibliography{SWbiblio}

\providecommand{\bysame}{\leavevmode\hbox to3em{\hrulefill}\thinspace}
\providecommand{\MR}{\relax\ifhmode\unskip\space\fi MR }
\providecommand{\MRhref}[2]{%
  \href{http://www.ams.org/mathscinet-getitem?mr=#1}{#2}
}
\providecommand{\href}[2]{#2}

\end{document}